\documentclass[reqno,11pt]{amsart}

\usepackage{verbatim}
\usepackage{color}
\usepackage{graphicx}
\usepackage[latin1]{inputenc}
\usepackage{vmargin}
\usepackage{amssymb}
\usepackage{boldline}
\usepackage{float}
\usepackage{algorithmic}
\usepackage{algorithm}
\usepackage{subfigure}
\usepackage{mathrsfs}
\usepackage{wrapfig}
\usepackage{setspace}
\usepackage{stmaryrd}
\usepackage{multirow}
\usepackage{rotating}
\usepackage{bm}

\graphicspath{{plots/}}

\renewcommand{\(}{\left(}
\renewcommand{\)}{\right)}
\renewcommand{\[}{\left[}
\renewcommand{\]}{\right]}

\newcommand{\dnorm}[2]{\left\| #1 \right\|_{#2}}

\renewcommand{\dfrac}[2]{\displaystyle{\frac{#1}{#2}}}
\newcommand{\dint}[2]{\displaystyle{\int_{#1}^{#2}}}
\newcommand{\dcup}[2]{\displaystyle{\bigcup_{#1}^{#2}}}

\newcommand{\dsum}[2]{\displaystyle{\sum_{#1}^{#2}}\,}
\newcommand{\dmax}[1]{\displaystyle{\max_{#1}}}

\newcommand{\supp}{\text{supp}}
\newcommand{\qspan}{\text{span}}

\newcommand{\qa}[3]{a_{#1}\( #2,\, #3\)}
\newcommand{\qaD}[3]{a^D_{#1}\( #2,\, #3\)}
\newcommand{\qaS}[3]{a^S_{#1}\( #2,\, #3\)}
\newcommand{\qaB}[3]{a^B_{#1}\( #2,\, #3\)}
\newcommand{\qaSE}[3]{a^{SE}_{#1}\( #2,\, #3\)}

\newcommand{\qmsum}[3]{m_{#1}\( #2,\, #3\)}
\newcommand{\qmsumSE}[3]{m^{SE}_{#1}\( #2,\, #3\)}
\newcommand{\qmsumtSE}[3]{\widetilde{m}^{SE}_{#1}\( #2,\, #3\)}
\newcommand{\qmsumD}[3]{m^D_{#1}\( #2,\, #3\)}

\newcommand{\qmsumS}[3]{m^S_{#1}\( #2,\, #3\)}
\newcommand{\qmsumB}[3]{m^B_{#1}\( #2,\, #3\)}
\newcommand{\Hess}[1]{\mathcal{H}(#1)}

\definecolor{gray}{rgb}{0.6,0.6,0.6}

\theoremstyle{plain}
\newtheorem{theorem}{Theorem}
\newtheorem{lemma}[theorem]{Lemma}

\theoremstyle{definition}

\theoremstyle{remark}
\newtheorem{remark}[theorem]{Remark}

\makeatletter \@addtoreset{equation}{section} \makeatother
\makeatletter \@addtoreset{theorem}{section} \makeatother
\makeatletter \@addtoreset{figure}{section} \makeatother
\makeatletter \@addtoreset{table}{section} \makeatother

\makeatletter \@namedef{subjclassname@2010}{
\textup{2010} Mathematics Subject Classification} \makeatother

\def \qAft{($\widetilde{\text{A}}$)}
\def \qn{n}
\def \qnI{n_I}

\def \qnbx{n_{\bm{x}}}
\def \ql{l}
\def \qL{L}

\def \qLt{\widetilde{L}}

\def \qkappa{\kappa}
\def \qkappainv{\kappa^{-1}}
\def \qkappamax{\kappa_{\text{max}}}
\def \qkappamin{\kappa_{\text{min}}}

\def \qmu{\mu}
\def \laplace{\Delta}

\def \qbe{\bm{e}}

\def \qbu{\bm{u}}

\def \qu{u}
\def \qphi{\phi}
\def \qpsi{\psi}

\def \qvarphi{\varphi}
\def \qlambda{\lambda}
\def \qlambdat{\widetilde{\lambda}}

\def \qbv{\bm{v}}
\def \qv{v}

\def \qp{p}

\def \qq{q}
\def \qi{i}
\def \qI{I}
\def \qk{k}

\def \qj{j}

\def \qepsilon{\epsilon}

\def \qVscr{\mathscr{V}}
\def \qVscrcSE{\mathscr{V}^c_{SE}}
\def \qVscrtH{\widetilde{\mathscr{V}}_H}
\def \qVscrtcSE{\widetilde{\mathscr{V}}^c_{SE}}

\def \qVscrcS{\mathscr{V}^c_S}

\def \qTcalH{\mathcal{T}_H}

\def \qVscrH{\mathscr{V}_{H}}

\def \qTH{T}

\def \qT{T}

\def \qE{E}

\def \qH{H}

\def \qf{f}
\def \qbf{\bm{f}}

\def \qOmega{\Omega}

\def \qOmegas{\Omega^{s}}

\def \qOmegasc{\overline{\Omega}^s}
\def \qOmegap{\Omega^{p}}
\def \qOmegapc{\overline{\Omega}^p}
\def \qOmegat{\widetilde{\Omega}}
\def \qdOmega{\partial\Omega}
\def \qOmegac{\overline{\Omega}}
\def \qomega{\omega}
\def \qomegac{\overline{\omega}}
\def \qdbx{\,d\bm{x}}

\def \grad{\nabla}
\def \div{\nabla\cdot}
\def \curl{\nabla\hspace{-.5ex}\times\hspace{-.5ex}}
\def \qbn{\bm{n}}

\def \qalpha{\alpha}
\def \qgamma{\gamma}

\def \qbeta{\beta}
\def \qtaulambda{\tau_{\lambda}}

\def \qFcal{\mathcal{F}}

\def \qC{C}

\def \qbx{\bm{x}}

\def \qXcal{\mathcal{X}}
\def \qx{x}
\def \qxi{\xi}
\def \qxit{\widetilde{\xi}}

\def \qbR{\mathbb{R}}
\def \qbN{\mathbb{N}}

\title[Robust DD Preconditioners for Abstract SPD Operators]
{Robust Domain Decomposition Preconditioners for Abstract Symmetric Positive Definite Bilinear Forms}

\author{Y.~Efendiev} 
\address{Dept. Mathematics, Texas A\&M University, College Station, Texas, 77843, USA 
({\tt efendiev@math.tamu.edu})}

\author{J.~Galvis}
\address{Dept. Mathematics, Texas A\&M University, College Station, Texas, 77843, USA 
({\tt jugal@math.tamu.edu})}

\author{R.~Lazarov}
\address{Dept. Mathematics, Texas A\&M University, College Station, Texas, 77843, USA 
({\tt lazarov@math.tamu.edu})}

\author{J.~Willems}
\address{Radon Institute for Computational and Applied Mathematics (RICAM), Altenberger Strasse 69, 4040 Linz, Austria
({\tt joerg.willems@ricam.oeaw.ac.at})}

\begin{document}

\begin{abstract}
An abstract framework for constructing stable decompositions of the spaces corresponding to general symmetric positive definite problems into ``local'' subspaces and a global ``coarse'' space is developed. Particular applications 
of this abstract framework include practically 
important problems in porous media applications such as: the scalar elliptic (pressure) equation and the stream function formulation of its mixed form, Stokes' and Brinkman's equations. The constant in the corresponding abstract energy estimate is shown to be robust with respect to mesh parameters as well as the contrast, which is defined as the ratio of high and low values of the conductivity (or permeability). The derived stable decomposition 
allows to construct additive overlapping Schwarz iterative methods with condition numbers uniformly bounded with respect to the contrast and mesh parameters. 
The coarse spaces are obtained by patching together the eigenfunctions corresponding to the smallest eigenvalues of certain local problems. A detailed analysis of the abstract setting is provided.
The proposed decomposition builds on a method of 
Efendiev and Galvis [{\it Multiscale Model. Simul.}, 8 (2010), pp. 1461--1483] developed 
for second order scalar elliptic problems with high contrast. 
Applications to the finite element discretizations of the second order
elliptic problem in Galerkin and mixed formulation, the Stokes equations, and Brinkman's problem are presented. A 
number of numerical experiments for these problems in two spatial dimensions are provided. 
\end{abstract}


\keywords{domain decomposition, robust additive Schwarz preconditioner, spectral coarse spaces, high contrast, Brinkman's problem, multiscale problems}

\subjclass[2010]{65F10, 65N20, 65N22, 65N30, 65N55}

\maketitle
\section{Introduction}
\label{sec:intro}
Symmetric positive definite operators appear 
in the modeling of a variety of problems from environmental and engineering sciences, e.g.\  heat conduction in industrial (compound) media or fluid flow in natural subsurfaces. Two main challenges arising in the numerical solution of these problems are (1) the problem size due to spatial scales and (2) high-contrast due to large variations in the physical problem parameters. The latter e.g.\ concerns disparities in the thermal conductivities of the constituents of compound media or in the permeability field of porous media. These variations frequently range over several orders of magnitude leading to high condition numbers of the corresponding discrete systems. Besides variations in  physical parameters, the  discretization parameters (e.g.\ mesh size)  also lead to large condition numbers of the respective discrete systems.

Since in general high condition numbers entail  poor performances of iterative solvers such as conjugate gradients (CG), the design of preconditioners addressing the size of the problem and the magnitude in variations of the problem parameters has received lots of attention in the numerical analysis community. The commonly used approaches include domain decomposition methods (cf.\ e.g.\ \cite{Mat08,SBG96,TW05}) and multilevel/multigrid algorithms (cf.\ e.g.\ \cite{Bra93,Hac03,Vassilevski_book_08}). For certain classes of problems, including the scalar elliptic case, these methods are  successful in making the condition number of the preconditioned system independent of the size of the problem. However, the design of preconditioners that are robust with respect to variations in physical parameters is more challenging.

Improvements in the standard domain decomposition methods for the scalar elliptic equation, $-\div(\qkappa(\qbx)\grad \qphi)=\qf$, with a highly variable conductivity $\qkappa(\qbx)$ were made in the case of special arrangements of the highly conductive regions with respect to the coarse cells. The construction of preconditioners for these problems has been extensively studied in the last three decades (see e.g.\ \cite{SarkisDryja96,IvanRob,MandelBrezina96,Mat08,TW05}).
In the context of domain decomposition methods, one can consider overlapping and nonoverlapping methods. 
It was shown that nonoverlapping domain decomposition methods converge independently of the contrast  (e.g. \cite{DryjaWidlund02,Nepom-91} and \cite[Sections~6.4.4 and 10.2.4]{TW05}) when conductivity variations within coarse regions are bounded. 
The condition number bound for the preconditioned linear system using a two level overlapping domain decomposition method scales with the contrast, defined as 
\begin{equation}
\label{contrastportion}
\max_{\qbx\in \Omega} \qkappa(\qbx) / \min_{\qbx\in \Omega}\qkappa(\qbx) 
\end{equation}
where  $\Omega$ is the domain.
The overall condition number estimate also involves the ratio $H/\delta$, where $H$ is the coarse-mesh size and
$\delta$ is the size of the overlap region. The estimate with respect to the ratio $H/\delta$ can be improved with the help of the small overlap trick (e.g.\ \cite{IvanRob,TW05}). In this paper we  focus on improving the contrast-dependent portion (given by (\ref{contrastportion})) of the condition number. A generous overlap will be used in our methods.

Classical arguments to estimate 
the preconditioned condition number 
of a two level overlapping  domain decomposition method use 
weighted Poincar\'e inequalities
of the form 
\begin{equation}\label{eq:requiredinequality}
\int_{\qomega} \qkappa|\grad \qxi|^2(\qpsi- 
I_0^{\qomega}\qpsi )^2\qdbx
\leq \qC\int_{\qomega}  \qkappa |\grad \qpsi|^2\qdbx, 
\end{equation}
where $\qomega$ is a local subdomain in the global domain $\qOmega$,  $\qxi$  
is  a partition of unity function subordinate to  $\omega$, and 
$\qpsi\in H^1(\omega)$. The operator $I_0^\omega \qpsi$ is a local representation of the 
function $\qpsi$ in the coarse space. The constant $\qC$ above 
appears in the final bound for the condition number of the operator. 
Many of the classical arguments  that analyze 
overlapping domain decomposition methods for high contrast problems 
assume that
$\frac{\max_{\qbx\in \omega} \kappa(\qbx) }
       {\min_{\qbx\in \omega}\kappa (\qbx)}$ is bounded.
When only $|\nabla \qxi|^2$ is bounded, inequality 
(\ref{eq:requiredinequality}) can be obtained 
from a  weighted Poincar\'e  inequality whose constant is 
independent of the contrast. This weighted Poincar\'e 
inequality 
is not always valid, so a number of works were successful in addressing
the question of when it holds.
In \cite{SarkisThesis}, it was proven that the weighted Poincar\'e  inequality holds for quasi-monotonic coefficients. The author obtains robust 
preconditioners for the case of quasi-monotonic coefficients. 
We note that recently, the concept of quasi-monotonic 
coefficient has been generalized 
in \cite{RobClemens2,RobClemens1} where the authors 
analyze nonoverlapping FETI methods.  
Other approaches use special partitions of unity 
such that the ``pointwise energy``  $\qkappa(\qbx)|\grad \qxi|^2$
in (\ref{eq:requiredinequality})
is bounded. In \cite{IvanRob,IvanRobEnergyMin}, the analysis 
in \cite{TW05} has been extended to obtain explicit bounds involving 
the quantity $\qkappa(\qbx)|\grad \qxi|^2$. 
It has been shown that if the coarse space basis 
functions are constructed 
properly, then $\qkappa(\qbx)|\grad \qxi|^2$ remains bounded 
for all basis and partition of unity functions 
used, and then, the classical Poincar\'e inequality can be applied in the analysis. 
In particular, two main sets of coarse basis functions have been used: (1) multiscale finite element functions
with various boundary conditions (see \cite{EfendievBook,IvanRob,HouWuCai})
and (2) energy minimizing or trace minimizing  functions (see 
\cite{IvanRobEnergyMin,Ludmil} and references therein).
Thus, robust overlapping domain 
decomposition methods can be constructed for the case when 
the high-conductivity regions are isolated  islands. These methods  use
one coarse basis function per coarse node.

Recently, new coarse basis functions have been  proposed in \cite{EG2,EG1}. 
The construction of these coarse basis functions uses local generalized eigenvalue problems.
The resulting methods can handle a general class of heterogeneous
coefficients with high contrast. A main step in this construction is 
to identify those initial multiscale basis functions
that are used to compute a  weight function for the eigenvalue
problem. These initial multiscale basis functions are designed to capture the effects that can be localized within coarse blocks. They are further  complemented
using generalized eigenfunctions that account for features
of the solution that cannot be localized. The idea of using local and global 
eigenvectors to construct coarse spaces within two-level
and multi-level techniques has been used before (e.g.\ \cite{VCF03,SarkisEnriched}). However, these authors did not 
study the convergence with respect to physical parameters, such
as high contrast and physical parameter variation, and did not   use
 generalized eigenvalue problems to achieve small dimensional
coarse spaces.

In many applications, the discretization technique is chosen so that it preserves the 
essential physical properties. For example, mixed finite element methods are
often used in flow equations to obtain locally mass conservative velocity fields.
In a number of flow applications, high-conductivity regions
need to be represented with flow equations, such as Navier-Stokes' or Stokes'
equations due to high porosity. Such complex systems can be described by Brinkman's equations that 
may require a special stable discretization. The method proposed in  \cite{EG2,EG1} cannot be easily applied to  vector problems and more complicated
discretization methods. More precisely, it requires 
a proper eigenvalue problem for each particular differential equation.

In this paper, we extend
the framework proposed in  \cite{EG2,EG1} to general symmetric bilinear forms.
The resulting analysis  can be applied to a wide variety of differential equations  that are important in practice. A key in designing robust preconditioners is a stable decomposition 
of the global function space into local and coarse subspaces (see \cite{TW05}).
This is the main focus of our paper. We develop an abstract framework that allows deriving
a generalized eigenvalue problem for the construction of the coarse spaces and stable decompositions. In particular, some explicit bounds are obtained
for the stability constant of this  decomposition.
In the scalar elliptic case,  the analysis presented here leads to  generalized eigenvalue problems
that differ from those studied in \cite{EG2,EG1}. 

We consider the application of this abstract framework to Darcy's equation as well as to Brinkman's equations. Brinkman's equations can be viewed as a generalization
of Darcy's equation that allows both  Darcy and Stokes regions in the flow. Because Darcy's equation
is obtained under the assumption of slow flow, Brinkman's equation is inherently high-contrast.
It combines high flow described by Stokes' equations and slow flow described by Darcy's equation.
In fact, the use of high conductivities in Darcy's 
equation may be associated to free flow or  high porosity regions that 
are often described by Brinkman's equations (see \cite{Bri47_1}).
The proposed general framework can be applied to the construction of robust coarse spaces for Brinkman's equations. These coarse spaces are constructed by passing to the stream function formulation. 

In the paper at hand, we also discuss some coarse space dimension reduction techniques within our abstract framework.
The dimension reduction is achieved by choosing initial multiscale basis functions that capture as much subgrid information as possible. We discuss how the choice of various initial multiscale basis functions affects the condition number of the preconditioned system.

To test the developed theory we consider several numerical examples. In our first set of
numerical examples, we study elliptic equations with highly variable 
coefficients. We use both piecewise bilinear and  multiscale basis functions as an initial coarse space.
In the latter choice, they capture the effects of isolated
high-conductivity regions. Our numerical results show that in both
cases one obtains robust preconditioners. The use of multiscale
basis functions as an initial coarse space allows substantial dimension reduction
for problems with many small isolated inclusions.
Both Galerkin and mixed formulations are studied in this paper.
The next set of numerical results are on Brinkman's equations.
These equations are discretized using $H(div)$-conforming 
Discontinuous Galerkin methods. The numerical results show that the number of
iterations is independent of the contrast. In our final numerical example, we
consider a complex geometry with highly-variable coefficients without apparent separation of high and low conductivity regions. The numerical results show that using multiscale initial basis functions,
one can obtain robust preconditioners with small coarse dimensional spaces.  

The paper is organized as follows. In Section~\ref{sec:framework}, 
problem setting and notation are introduced. In Section~\ref{sec:stable_decomposition}, the abstract analysis of the stable decomposition is presented. Section~\ref{sec:applications} is dedicated 
to applications of the abstract framework to the Galerkin and mixed formulation 
of Darcy's equation, Stokes' equations, and Brinkman's equations.
In Section~\ref{sec:dim_coarse}, we discuss the dimension reduction 
of the coarse space. Representative numerical results are presented in
Section~\ref{sec:numerics}.

\section{Problem Setting and Notation}
\label{sec:framework}
Let $\qOmega\subset\qbR^\qn$ be a bounded polyhedral domain, and let $\qTcalH$ be a quasiuniform quadrilateral ($\qn=2$) or hexahedral ($\qn=3$) triangulation of $\qOmega$ with mesh-parameter $\qH$. Let $\qXcal=\{\qbx_\qj\}_{\qj=1}^{\qnbx}$ be the set of nodes of $\qTcalH$, and for each $\qbx_\qj\in\qXcal$ we set 
$$
\qOmega_\qj:=interior\(\dcup{}{} \{\overline{\qT} \,|\, \qT\in \qTcalH,\    \qbx_\qj \in \overline{\qT} \}\),
$$ i.e., 
$\qOmega_\qj$ 
is the union of all cells surrounding $\qbx_\qj$ (see Figure~\ref{fig:domain}).
\begin{figure}
\resizebox{.25\textwidth}{!}{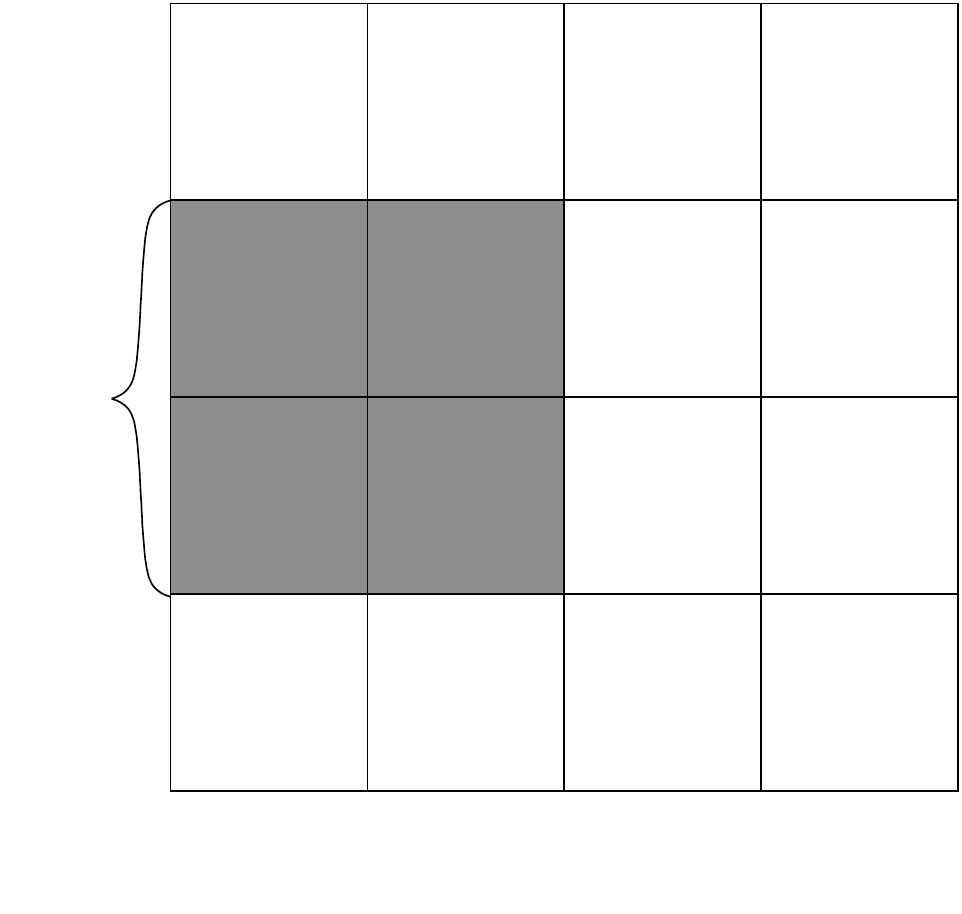}
\caption{Triangulation $\qTcalH$.}
\label{fig:domain}
\end{figure}

For a suitable separable Hilbert space $\qVscr_0=\qVscr_0(\qOmega)$ of functions defined on $\qOmega$ and for any subdomain $\qomega\subset\qOmega$ we set $\qVscr(\qomega):=\{\qphi|_{\qomega}\, |\, \qphi\in\qVscr_0\}$ and we consider a family of symmetric positive semi-definite bounded bilinear forms 
$$\qa{\qomega}{\cdot}{\cdot}:\,\(\qVscr(\qomega),\,\qVscr(\qomega)\)\rightarrow \qbR.$$ 
For the case $\qomega = \qOmega$ we drop the subindex, i.e., $\qa{}{\cdot}{\cdot} = \qa{\qOmega}{\cdot}{\cdot}$ and we additionally assume that $\qa{}{\cdot}{\cdot}$ is positive definite. For ease of notation we write $\qa{\qomega}{\qphi}{\qpsi}$ instead of $\qa{\qomega}{\qphi|_{\qomega}}{\qpsi|_{\qomega}}$ for any $\qphi,\, \psi\in\qVscr_0$.
Furthermore, we assume that for any $\qphi\in\qVscr_0$ and any family of pairwisely disjoint subdomains
$\{\qomega_\qj\}_{\qj=1}^{\qn_\qomega}$ with $\dcup{\qj=1}{\qn_\qomega}\qomegac_\qj=\qOmegac$ 
\begin{equation}
\label{eq:sum_bilin}
\qa{}{\qphi}{\qphi}=\sum_{\qj=1}^{\qn_\qomega}\qa{\qomega_\qj}{\qphi} {\qphi}.
\end{equation}

Now, the goal is to construct a ``coarse'' subspace $\qVscrH=\qVscrH(\qOmega)$ of $\qVscr_0$ with the following property: for any $\qphi\in\qVscr_0$ there is a representation
\begin{subequations}
\label{eq:stable_decom}
\begin{equation}
\label{eq:stable_decom_1}
\qphi = \dsum{\qj=0}{\qnbx}\qphi_\qj \text{ with } \qphi_0\in\qVscrH,\text{ and } \qphi_{\qj}\in\qVscr_0(\qOmega_\qj) \text{ for } \qj=1,\ldots,\qnbx
\end{equation}
such that 
\begin{equation}
\label{eq:stable_decom_2}
\dsum{\qj=0}{\qnbx}\qa{}{\qphi_\qj}{\qphi_\qj} \le \qC \qa{}{\qphi}{\qphi},
\end{equation}
\end{subequations}
where $\qVscr_0(\qOmega_\qj)\subset \qVscr(\qOmega_\qj)$ is a subspace such that $\qa{\qOmega_\qj}{\cdot}{\cdot}:\, (\qVscr_0(\qOmega_\qj),\qVscr_0(\qOmega_\qj))\rightarrow \qbR$ is positive definite. Note that $\qa{\qOmega_\qj}{\cdot}{\cdot}:\, (\qVscr(\qOmega_\qj),\qVscr(\qOmega_\qj))\rightarrow \qbR$ is in general only positive {\bf semi}-definite.

By e.g.\ \cite[Section~2.3]{TW05} we know that the additive Schwarz preconditioner corresponding to \eqref{eq:stable_decom_1} yields a condition number that only depends on the constant $\qC$ in \eqref{eq:stable_decom_2} and the maximal number of overlaps of the subdomains $\qOmega_\qj$, $\qj=1,\ldots,\qnbx$. Thus, we would like to ``control'' the constant $\qC$ and keep the dimension of $\qVscrH$ ``as small as possible''.

For the construction of $\qVscrH$ we need more notation.
Let $\{\qxi_\qj\}_{\qj=1}^{\qnbx}:\,\qOmega\rightarrow[0,\,1]$ be a partition of unity subordinate to $\{\qOmega_\qj\}_{\qj=1}^{\qnbx}$ such that $\supp(\qxi_\qj)=\qOmegac_\qj$ and for any $\qphi\in\qVscr_0$  we have $\qxi_\qj\qphi\in\qVscr_0$, $(\qxi_\qj\qphi)|_{\qOmega_\qj}\in\qVscr_0(\qOmega_\qj)$,  $\qj=1,\ldots,\qnbx$.
Using this notation we may for any $\qi,\,\qj=1,\ldots,\qnbx$ define the following symmetric bilinear form:
\begin{equation}
\label{eq:def_msum}
\begin{array}{rcl}
\qmsum{\qOmega_\qj}{\cdot}{\cdot}:\, (\qVscr(\qOmega_\qj),\qVscr(\qOmega_\qj))&\rightarrow& \qbR\\
\qmsum{\qOmega_\qj}{\qphi}{\qpsi} & := & \dsum{\qi\in\qI_\qj}{}\qa{\qOmega_\qj}{\qxi_\qj\qxi_\qi \qphi}{\qxi_\qj\qxi_\qi \qpsi}
\end{array}
\end{equation}
where $\qI_\qj:=\{\qi = 1,\ldots,\qnbx\, |\, \qOmega_\qi\cap\qOmega_\qj \ne \emptyset\}$. Let $\qnI:=\dmax{\qj=1,\ldots,\qnbx}\#\qI_\qj$ denote the maximal number of overlaps of the $\qOmega_\qj$'s. To ease the notation, as we did for the bilinear form $\qa{}{\cdot}{\cdot}$, we write $\qmsum{\qOmega_\qj}{\qphi}{\qpsi}$ instead of $\qmsum{\qOmega_\qj}{\qphi|_{\qOmega_\qj}}{\qpsi|_{\qOmega_\qj}}$ for any $\qphi,\, \psi\in\qVscr_0$.

Due to our assumptions on $\{\qxi_\qj\}_{\qj=1}^{\qnbx}$ we have that \eqref{eq:def_msum} is well-defined. Also note, that since $\supp(\qxi_\qj)=\qOmegac_\qj$ we have $\qxi_\qj\qphi|_{\qOmega_\qj} \equiv 0 \Leftrightarrow \qphi|_{\qOmega_\qj}\equiv 0$, which implies that $\qmsum{\qOmega_\qj}{\cdot}{\cdot}:\, (\qVscr(\qOmega_\qj),\qVscr(\qOmega_\qj))\rightarrow \qbR$ is positive definite.

Now for  $\qj=1,\ldots,\qnbx$ we consider the generalized eigenvalue problems:
Find $(\qlambda^\qj_\qi,\, \qvarphi^\qj_\qi)\in (\qbR,\, \qVscr(\qOmega_\qj))$ such that
\begin{equation}
\label{eq:gen_eigenval}
\qa{\qOmega_\qj}{\qpsi}{\qvarphi^\qj_\qi} = \qlambda^\qj_\qi\qmsum{\qOmega_\qj}{\qpsi}{\qvarphi^\qj_\qi},\quad \forall \qpsi\in\qVscr(\qOmega_\qj).
\end{equation}
Without loss of generality we assume that the eigenvalues are ordered, i.e., $0\le\qlambda^\qj_1 \le \qlambda^\qj_2 \le \ldots \le \qlambda^\qj_\qi \le \qlambda^\qj_{\qi +1} \le \ldots$.

It is easy to see that any two eigenfunctions corresponding to two distinct eigenvalues are $\qa{\qOmega_\qj}{\cdot}{\cdot}$- and $\qmsum{\qOmega_\qj}{\cdot}{\cdot}$-orthogonal. By orthogonalizing the eigenfunctions corresponding to the same eigenvalues we can thus, without loss of generality, assume that all computed eigenfunctions are pairwisely $\qa{\qOmega_\qj}{\cdot}{\cdot}$- and $\qmsum{\qOmega_\qj}{\cdot}{\cdot}$-orthogonal. Now, every function in $\qVscr(\qOmega_\qj)$ has an expansion with respect to the eigenfunctions of \eqref{eq:gen_eigenval}. This is the reason why the generalized eigenproblem is posed with respect to $\qVscr(\qOmega_\qj)$ as opposed to $\qVscr_0(\qOmega_\qj).$

For $\qphi\in\qVscr_0$ let $\qphi^\qj_0$ be the $\qmsum{\qOmega_\qj}{\cdot}{\cdot}$-orthogonal projection of $\qphi|_{\qOmega_\qj}$ onto the first $\qL_\qj$ eigenfunctions of \eqref{eq:gen_eigenval}, where $\qL_\qj\in\qbN_0$ is some non-negative integer, i.e.,
\begin{equation}
\label{eq:orth_proj}
\qmsum{\qOmega_\qj}{\qphi-\qphi^\qj_0}{\qvarphi^\qj_\qi} = 0, \quad \forall\, \qi=1,\ldots,\qL_\qj.
\end{equation}
If $\qL_\qj=0$, we set $\qphi^\qj_0\equiv 0$.
We assume that for any $\qj=1,\ldots,\qnbx$ and a sufficiently small ``threshold`` $\qtaulambda^{-1}>0$ we may choose $\qL_\qj$ such that $\qlambda^\qj_{\qL_\qj + 1} \ge \qtaulambda^{-1}$. Since any function in $\qVscr(\qOmega_\qj)$ has an expansion with respect to the eigenfunctions of \eqref{eq:gen_eigenval} we can now choose $\qalpha_\qi\in\qbR$ so that
$$\qphi|_{\qOmega_\qj}-\qphi^\qj_0 = \dsum{\qi>\qL_\qj}{}\qalpha_\qi \qvarphi^\qj_\qi.$$
Then we observe that
\begin{equation}
\label{eq:est_inv_eigen}
\begin{array}{rclr}
\qmsum{\qOmega_\qj}{\qphi-\qphi^\qj_0}{\qphi-\qphi^\qj_0}&=& \qmsum{\qOmega_\qj}{\dsum{\qi>\qL_\qj}{}\qalpha_\qi \qvarphi^\qj_\qi}{\dsum{\qi>\qL_\qj}{}\qalpha_\qi \qvarphi^\qj_\qi}&\\
&=&
\dsum{\qi>\qL_\qj}{}\qmsum{\qOmega_\qj}{\qalpha_\qi \qvarphi^\qj_\qi}{\qalpha_\qi \qvarphi^\qj_\qi}& \text{(by orthogonality)}\\
&=&
\dsum{\qi>\qL_\qj}{}\dfrac{1}{\qlambda^\qj_\qi}\qa{\qOmega_\qj}{\qalpha_\qi \qvarphi^\qj_\qi}{\qalpha_\qi \qvarphi^\qj_\qi} & \text{(by \eqref{eq:gen_eigenval})}\\
&\le&
\dfrac{1}{\qlambda^\qj_{\qL_\qj +1}} \dsum{\qi>\qL_\qj}{}\qa{\qOmega_\qj}{\qalpha_\qi \qvarphi^\qj_\qi}{\qalpha_\qi \qvarphi^\qj_\qi}&\\
&\le&
\qtaulambda\, \qa{\qOmega_\qj}{\qphi-\qphi^\qj_0}{\qphi-\qphi^\qj_0}\\
&\le&
\qtaulambda\, \qa{\qOmega_\qj}{\qphi}{\qphi}.
\end{array}
\end{equation}

\section{Coarse Spaces yielding Robust Stable Decompositions}
\label{sec:stable_decomposition}

With these preliminaries we are now able to define a decomposition described in \eqref{eq:stable_decom}: First, we specify the coarse space by 
\begin{equation}\label{eq:coarse}
\qVscrH:=\qspan\{\qxi_\qj\qvarphi^\qj_\qi\, |\, \qj=1,\ldots,\qnbx \text{ and } \qi=1,\ldots,\qL_\qj\}.
\end{equation}
Then, for any $\qphi\in\qVscr$ let 
\begin{subequations}
\label{eq:stab_decom_specific}
\begin{equation}
\label{eq:stab_decom_coarse}
\qphi_0:=\dsum{\qj=1}{\qnbx}\qxi_\qj \qphi^\qj_0\in \qVscrH,
\end{equation}
where $\qphi^\qj_0$ is chosen according to \eqref{eq:orth_proj}. For $\qj=1,\ldots,\qnbx$ define
\begin{equation}
\label{eq:stab_decom_fine}
\qphi_\qj:= (\qxi_\qj(\qphi-\qphi_0))|_{\qOmega_\qj}\in \qVscr_0(\qOmega_\qj) \quad \mbox{ so that } \quad 
\qphi = \dsum{\qj=0}{\qnbx} \qphi_\qj.
\end{equation}
\end{subequations}

Before analyzing this decomposition we summarize all assumptions using the notation above:
\begin{enumerate}
\item[(A1)] $\qa{\qomega}{\cdot}{\cdot}:\,(\qVscr(\qomega),\qVscr(\qomega))\rightarrow \qbR$ is symmetric positive semi-definite for any subdomain $\qomega\subset\qOmega$, and $\qa{}{\cdot}{\cdot}=\qa{\qOmega}{\cdot}{\cdot}$ is positive definite.
\item[(A2)] For any $\qphi\in\qVscr_0$ and any pairwisely disjoint family of subdomains $\{\qomega_\qj\}_{\qj=1}^{\qn_\qomega}$ with $\dcup{\qj=1}{\qn_\qomega}\qomegac_\qj=\qOmegac$ we have $\qa{}{\qphi}{\qphi}=\dsum{\qj=1}{\qn_\qomega}\qa{\qomega_\qj}{\qphi}{\qphi}$. 
(This implies that $\qa{}{\qphi}{\qphi}\le \dsum{\qj=1}{\qnbx}\qa{\qOmega_\qj}{\qphi}{\qphi}$, since $\qa{\qomega}{\cdot}{\cdot}$ is semi-definite.)
\item[(A3)] 
$\qa{\qOmega_\qj}{\cdot}{\cdot}:\,(\qVscr_0(\qOmega_\qj) ,\qVscr_0(\qOmega_\qj))\rightarrow \qbR$ is positive definite for all $1\le \qj\le \qnbx$.
\item[(A4)] $\{\qxi_\qj\}_{\qj=1}^{\qnbx}:\,\qOmega\rightarrow[0,\,1]$ is a family of functions such that
\begin{enumerate}
\item[(a)] $\dsum{\qj=1}{\qnbx}\qxi_\qj\equiv 1$ on $\qOmega$;
\item[(b)] 
$\supp(\qxi_\qj)=\qOmegac_\qj$ for  $\qj=1,\ldots,\qnbx$;
\item[(c)] 
For $\qphi\in\qVscr_0$ we have $\qxi_\qj\qphi\in\qVscr_0$ and $(\qxi_\qj\qphi)|_{\qOmega_\qj}\in\qVscr_0(\qOmega_\qj)$ for $\qj=1,\ldots,\qnbx$.
\end{enumerate}
\item[(A5)] For 
a sufficiently small threshold $\qtaulambda^{-1}$ we may choose $\qL_\qj$ such that $\qlambda^\qj_{\qL_\qj + 1} \ge \qtaulambda^{-1}$ for all $\qj=1,\ldots,\qnbx$.
\end{enumerate}
As noted above, these assumptions imply in particular, that for any $1\le \qj\le \qnbx$, the bilinear form $\qmsum{\qOmega_\qj}{\cdot}{\cdot}:\,(\qVscr(\qOmega_\qj),\qVscr(\qOmega_\qj))\rightarrow \qbR$ is positive definite.

The following lemma establishes the required energy bound for the local contributions in decomposition \eqref{eq:stable_decom}:
\begin{lemma}
\label{prop:local_est}
Assume (A1)--(A5) hold. Then, for $\qphi\in\qVscr_0$  we have
\begin{equation}
\dsum{\qj=1}{\qnbx}\qa{}{\qphi_\qj}{\qphi_\qj}  \le \qC\, \qtaulambda\, \qa{}{\qphi}{\qphi},
\end{equation}
where $\qC$ only depends on $\qnI$ (the maximal number of overlaps of the subdomains $\{\qOmega_\qj\}_{\qj=1}^{\qnbx}$).
\end{lemma}
\begin{proof}
Observe that
$$
\begin{array}{rclr}
\dsum{\qj=1}{\qnbx}\qa{}{\qphi_\qj}{\qphi_\qj} 
&=&
\dsum{\qj=1}{\qnbx}\qa{\qOmega_\qj}{\qxi_\qj(\qphi-\qphi_0)}{\qxi_\qj(\qphi-\qphi_0)} & \text{(by \eqref{eq:stab_decom_fine})}\\
&=&
\dsum{\qj=1}{\qnbx}\qa{\qOmega_\qj}{\qxi_\qj\dsum{\qi=1}{\qnbx}\qxi_\qi(\qphi-\qphi^\qi_0)}{\qxi_\qj \dsum{\qi=1}{\qnbx}\qxi_\qi(\qphi-\qphi^\qi_0)} & \text{(by \eqref{eq:stab_decom_coarse})}\\
&=&
\dsum{\qj=1}{\qnbx}\qa{\qOmega_\qj}{\qxi_\qj\dsum{\qi\in\qI_\qj}{}\qxi_\qi(\qphi-\qphi^\qi_0)}{\qxi_\qj \dsum{\qi\in\qI_\qj}{}\qxi_\qi(\qphi-\qphi^\qi_0)}&\\
&\le&
\qnI \underbrace{\dsum{\qj=1}{\qnbx}\dsum{\qi\in\qI_\qj}{}
\qa{\qOmega_\qj}{\qxi_\qj\qxi_\qi(\qphi-\qphi^\qi_0)}{\qxi_\qj \qxi_\qi(\qphi-\qphi^\qi_0)}}_{=:\qE_1},&\\
\end{array}
$$
where in the last step we have used Schwarz' inequality together with $\#\qI_\qj\le\qnI$. Note, that we furthermore have
$$
\begin{array}{rclr}
\qE_1
&=&
\dsum{\qj=1}{\qnbx}\dsum{\qi\in\qI_\qj}{} \qa{\qOmega_\qi}{\qxi_\qj\qxi_\qi(\qphi-\qphi^\qi_0)}{\qxi_\qj \qxi_\qi(\qphi-\qphi^\qi_0)}&\\
&=&
\dsum{\qi=1}{\qnbx}\dsum{\qj\in\qI_\qi}{} \qa{\qOmega_\qi}{\qxi_\qj\qxi_\qi(\qphi-\qphi^\qi_0)}{\qxi_\qj \qxi_\qi(\qphi-\qphi^\qi_0)}&\\
&=&
\dsum{\qi=1}{\qnbx}\qmsum{\qOmega_\qi}{\qphi-\qphi^\qi_0}{\qphi-\qphi^\qi_0}&(\text{by \eqref{eq:def_msum}}).\\
\end{array}
$$
Thus, we obtain
\begin{equation}
\begin{array}{rclr}
\dsum{\qj=1}{\qnbx}\qa{}{\qphi_\qj}{\qphi_\qj} 
&\le&
\qnI \dsum{\qj=1}{\qnbx}\qmsum{\qOmega_\qj}{\qphi-\qphi^\qj_0}{\qphi-\qphi^\qj_0}&\\
&\le&
\qnI\, \qtaulambda\, \dsum{\qj=1}{\qnbx} \qa{\qOmega_\qj}{\qphi}{\qphi}&(\text{by \eqref{eq:est_inv_eigen}})\\
&\le&
\qnI^2\, \qtaulambda\, \qa{}{\qphi}{\qphi},\\
\end{array}
\end{equation}
where the last inequality holds due to (A2) and the fact that any point in $\qOmega$ simultaneously lies in at most $\qnI$ subdomains $\qOmega_\qj$, $\qj=1,\ldots,\qnbx$. 
\end{proof}

\begin{remark}
\label{rem:est_local}
Note, that by the proof of Lemma~\ref{prop:local_est} we in particular have
$$
\dsum{\qj=1}{\qnbx}\qa{\qOmega_\qj}{\qxi_\qj\dsum{\qi\in\qI_\qj}{}\qxi_\qi(\qphi-\qphi^\qi_0)}{\qxi_\qj \dsum{\qi\in\qI_\qj}{}\qxi_\qi(\qphi-\qphi^\qi_0)}
\le \qnI^2\, \qtaulambda\, \qa{}{\qphi}{\qphi}.
$$
\end{remark}

Now, we proceed with the necessary energy estimate for the coarse component $\qphi_0$ of the decomposition
\eqref{eq:stab_decom_coarse}.

\begin{lemma}
\label{prop:coarse_est}
Let (A1)--(A5) hold. Then, for $\qphi_0$ defined by \eqref{eq:stab_decom_coarse}  we have that
\begin{equation}
\label{eq:coarse_est}
\qa{}{\qphi_0}{\qphi_0} \le (2+\qC\, \qtaulambda)\, \qa{}{\qphi}{\qphi},
\end{equation}
where, as above, $\qC$ only depends on $\qnI$.
\end{lemma}
\begin{proof}
First, we note that
\begin{equation}
\label{eq:proof2_1}
\begin{array}{rclr}
\qa{}{\qphi_0}{\qphi_0}
&=&
\qa{}{\dsum{\qi=1}{\qnbx} \qxi_\qi \qphi_0^\qi}{\dsum{\qi=1}{\qnbx} \qxi_\qi \qphi_0^\qi}&\text{(by \eqref{eq:stab_decom_coarse})}\\
&=&
\qa{}{\dsum{\qi=1}{\qnbx} \qxi_\qi (\qphi_0^\qi-\qphi) +\qphi}{\dsum{\qi=1}{\qnbx} \qxi_\qi (\qphi_0^\qi -\qphi)+\qphi}& \text{(by (A4))}\\
&\le&
2 \underbrace{\qa{}{\dsum{\qi=1}{\qnbx} \qxi_\qi (\qphi_0^\qi-\qphi)}{\dsum{\qi=1}{\qnbx} \qxi_\qi (\qphi_0^\qi -\qphi)} }_{=:E_2} + 2\qa{}{\qphi}{\qphi},
\end{array}
\end{equation}
where we have used Schwarz' inequality.
Now, observe that
\begin{equation}
\label{eq:proof2_2}
\begin{array}{rclr}
E_2
&\le&
\dsum{\qj=1}{\qnbx}\qa{\qOmega_\qj}{\dsum{\qi\in\qI_\qj}{} \qxi_\qi (\qphi_0^\qi-\qphi)}{\dsum{\qi\in\qI_\qj}{} \qxi_\qi (\qphi_0^\qi-\qphi)} &\text{(by (A2))}\\
&=&
\dsum{\qj=1}{\qnbx}\qa{\qOmega_\qj}{(1-\qxi_\qj+\qxi_\qj)\dsum{\qi\in\qI_\qj}{} \qxi_\qi (\qphi_0^\qi-\qphi)}{(1-\qxi_\qj+\qxi_\qj)\dsum{\qi\in\qI_\qj}{} \qxi_\qi (\qphi_0^\qi-\qphi)}\\
&\le&
2 \underbrace{\dsum{\qj=1}{\qnbx}\qa{\qOmega_\qj}{(1-\qxi_\qj)\dsum{\qi\in\qI_\qj}{} \qxi_\qi (\qphi_0^\qi-\qphi)}{(1-\qxi_\qj)\dsum{\qi\in\qI_\qj}{} \qxi_\qi (\qphi_0^\qi-\qphi)}}_{=:E_3}\\
&&+
2 \dsum{\qj=1}{\qnbx}\qa{\qOmega_\qj}{\qxi_\qj\dsum{\qi\in\qI_\qj}{} \qxi_\qi (\qphi_0^\qi-\qphi)}{\qxi_\qj\dsum{\qi\in\qI_\qj}{} \qxi_\qi (\qphi_0^\qi-\qphi)}&\\[5ex]
&\le&
2 \qE_3 + 2\qnI^2\, \qtaulambda\, \qa{}{\qphi}{\qphi},& \text{(by Remark~\ref{rem:est_local})}
\end{array}
\end{equation}
where we have again used Schwarz' inequality.
Note, that
$$
\begin{array}{rclr}
E_3&=&
\dsum{\qj=1}{\qnbx}\qa{\qOmega_\qj}{\dsum{\ql\in\qI_\qj\backslash\{\qj\}}{}\qxi_\ql\dsum{\qi\in\qI_\qj}{} \qxi_\qi (\qphi_0^\qi-\qphi)}{\dsum{\ql\in\qI_\qj\backslash\{\qj\}}{}\qxi_\ql\dsum{\qi\in\qI_\qj}{} \qxi_\qi (\qphi_0^\qi-\qphi)}& \text{(by (A4))}\\
&\le&
\qnI\dsum{\qj=1}{\qnbx}\dsum{\ql\in\qI_\qj\backslash\{\qj\}}{}\qa{\qOmega_\qj}{\qxi_\ql\dsum{\qi\in\qI_\qj}{} \qxi_\qi (\qphi_0^\qi-\qphi)}{\qxi_\ql\dsum{\qi\in\qI_\qj}{} \qxi_\qi (\qphi_0^\qi-\qphi)} &\text{(by Schwarz' inequality).}\\
\end{array}
$$
Since
$$
\begin{array}{rcl}
\qa{\qOmega_\qj}{\qxi_\ql\dsum{\qi\in\qI_\qj}{} \qxi_\qi (\qphi_0^\qi-\qphi)}{\qxi_\ql\dsum{\qi\in\qI_\qj}{} \qxi_\qi (\qphi_0^\qi-\qphi)}
&=&
\qa{\qOmega_\ql}{\qxi_\ql\dsum{\qi\in\qI_\qj\cap\qI_\ql}{} \qxi_\qi (\qphi_0^\qi-\qphi)}{\qxi_\ql\dsum{\qi\in\qI_\qj\cap\qI_\ql}{} \qxi_\qi (\qphi_0^\qi-\qphi)}\\
&\le&
\qa{\qOmega_\ql}{\qxi_\ql\dsum{\qi\in\qI_\ql}{} \qxi_\qi (\qphi_0^\qi-\qphi)}{\qxi_\ql\dsum{\qi\in\qI_\ql}{} \qxi_\qi (\qphi_0^\qi-\qphi)},
\end{array}
$$
we thus have
\begin{equation}
\label{eq:proof2_4}
\begin{array}{rclr}
E_3
&\le&
\qnI\dsum{\qj=1}{\qnbx}\dsum{\ql\in\qI_\qj\backslash\{\qj\}}{}
\qa{\qOmega_\ql}{\qxi_\ql\dsum{\qi\in\qI_\ql}{} \qxi_\qi (\qphi_0^\qi-\qphi)}{\qxi_\ql\dsum{\qi\in\qI_\ql}{} \qxi_\qi (\qphi_0^\qi-\qphi)}\\
&=&
\qnI\dsum{\ql=1}{\qnbx}\dsum{\qj\in\qI_\ql\backslash\{\ql\}}{}
\qa{\qOmega_\ql}{\qxi_\ql\dsum{\qi\in\qI_\ql}{} \qxi_\qi (\qphi_0^\qi-\qphi)}{\qxi_\ql\dsum{\qi\in\qI_\ql}{} \qxi_\qi (\qphi_0^\qi-\qphi)}\\
&\le&
\qnI^2\dsum{\ql=1}{\qnbx}
\qa{\qOmega_\ql}{\qxi_\ql\dsum{\qi\in\qI_\ql}{} \qxi_\qi (\qphi_0^\qi-\qphi)}{\qxi_\ql\dsum{\qi\in\qI_\ql}{} \qxi_\qi (\qphi_0^\qi-\qphi)}&  \text{($\#(\qI_\ql\backslash\{\ql\})\le\qnI$)}\\
&\le&
\qnI^4\, \qtaulambda\, \qa{}{\qphi}{\qphi} &\text{(by Remark~\ref{rem:est_local})}.\\
\end{array}
\end{equation}

Combining \eqref{eq:proof2_1}, \eqref{eq:proof2_2}, and \eqref{eq:proof2_4} we obtain
$$
\qa{}{\qphi_0}{\qphi_0} \le \(2+4(\qnI^4+\qnI^2) \qtaulambda \)\qa{}{\qphi}{\qphi}.
$$
\end{proof}

Combining Lemmas~\ref{prop:local_est} and \ref{prop:coarse_est} we have the following:
\begin{theorem}
\label{cor:stable_decom}
Assume (A1)--(A5) hold. Then, the decomposition defined in \eqref{eq:stab_decom_specific} satisfies
\begin{equation}
\label{eq:est_stable_decom}
\dsum{\qj=0}{\qnbx}\qa{}{\qphi_\qj}{\qphi_\qj} \le (2+\qC\, \qtaulambda) \qa{}{\qphi}{\qphi},
\end{equation}
where $\qC$ only depends on $\qnI$.
\end{theorem}

\section{Applications}
\label{sec:applications}
To apply the theory developed in Sections~\ref{sec:framework} and \ref{sec:stable_decomposition} to some particular problem we need to verify assumptions (A1)--(A5). 
Once this is established, we can conclude that the condition number of the corresponding additive Schwarz preconditioned system has an upper bound which depends only on $\qnI$ and $\qtaulambda$. Thus, if we can show the validity of (A5) with $\qtaulambda$  chosen independently of certain problem parameters and mesh parameter $\qH$, the condition number will also be bounded independently of these problem parameters and the mesh parameter. 

In addition to verifying assumptions (A1)--(A5) we also have to make sure that the number of ``small'' eigenvalues in \eqref{eq:gen_eigenval}, i.e., those below the threshold $\qtaulambda^{-1}$, is a small in order for our method to be practically beneficial. Note, that the choice of $\qtaulambda^{-1}>0$ and thus the number of small eigenvalues is not unique. Nevertheless, for a certain choice of $\qtaulambda^{-1}$ and a given problem we may still aim to establish that the number of eigenvalues below the chosen threshold is uniformly bounded with respect to changes in specific problem parameters. In this case the additive Schwarz preconditioner corresponding to the stable decomposition \eqref{eq:stab_decom_specific} has a coarse space whose dimension is uniformly bounded with respect to these parameters.

\subsection{The Scalar Elliptic Case -- Galerkin Formulation}
\label{sec:ScalarElliptic}
As a first application of the abstract framework developed above, we consider the scalar elliptic equation
\begin{equation}
\label{eq:scalar_elliptic}
-\div ( \qkappa \grad \qphi ) = \qf, \quad  \qbx \in \qOmega, \text{ and } \quad \qphi=0, \quad  \qbx \in \qdOmega,
\end{equation}
where $\qkappa\in L^\infty(\qOmega)$ is a positive function, which may have large variation, $\qphi\in H^1_0(\qOmega)$, and $\qf\in L^2(\qOmega)$.
Note that we can always reduce to the case of homogeneous boundary conditions by introducing an appropriate right hand side.

With $\qVscr_0:=H^1_0(\qOmega)$, the variational formulation corresponding to \eqref{eq:scalar_elliptic} is: 
Find $\qphi\in\qVscr_0$ such that for all $\qpsi\in\qVscr_0$ 
$$
\qaSE{}{\qphi}{\qpsi}:=\dint{\qOmega}{}\qkappa(\qbx) \grad\qphi\cdot\grad\qpsi \qdbx = \dint{\qOmega}{}\qf \qpsi \qdbx.
$$
Let $\qVscr_0(\qOmega_\qj):= H^1_0(\qOmega_\qj)\subset \qVscr_0|_{\qOmega_\qj}$ for any $\qj=1,\ldots,\qnbx$.
Choosing $\{\qxi_\qj\}_{\qj=1}^{\qnbx}$ the Lagrange finite element functions of degree one corresponding to $\qTcalH$, we readily see that for the scalar elliptic case, i.e., when setting $\qa{}{\cdot}{\cdot}=\qaSE{}{\cdot}{\cdot}$, assumptions (A1)--(A4) are satisfied.

Let us for now assume that $\qkappa$ assumes only two values. More precisely,
$$
\qkappa(\qbx)=
\left\{
\begin{array}{ll}
\qkappamin &\text{in }\qOmegas\\
\qkappamax &\text{in }\qOmegap,
\end{array}
\right. \text{ with } \qOmegasc\cup\qOmegapc=\qOmegac \text{ and } \qkappamax\gg\qkappamin>0.
$$
Without loss of generality, we may take $\qkappamin = 1$.
Now we wish to establish the existence of $\qtaulambda^{-1}$ such that the number of eigenvalues below this threshold is independent of the contrast $\qkappamax/\qkappamin$, which is the problem parameter of interest in this situation.

By the well-known min-max principle \cite{Reed_Simon}, we know that the $\qi$-th eigenvalue of \eqref{eq:gen_eigenval} is given by
\begin{equation}
\label{eq:min_max}
\qlambda^\qj_\qi=\min_{\qVscr_\qi(\qOmega_\qj)}\max_{\qpsi\in\qVscr_\qi(\qOmega_\qj)}\dfrac{\qaSE{\qOmega_\qj}{\qpsi}{\qpsi}}{\qmsumSE{\qOmega_\qj}{\qpsi}{\qpsi}},
\end{equation}
where $\qVscr_\qi(\qOmega_\qj)$ is any $\qi$-dimensional subspace of $\qVscr(\qOmega_\qj)$. (Here $\qmsumSE{\qOmega_\qj}{\cdot}{\cdot}$ is defined according to \eqref{eq:def_msum} with $\qa{}{\cdot}{\cdot}$ replaced by $\qaSE{}{\cdot}{\cdot}$.)

\begin{figure}
\resizebox{.25\textwidth}{!}{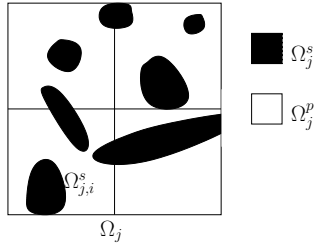}
\caption{Subdomain with connected components of $\qOmegas_\qj$. In the present configuration $\qL_\qj=7$.}
\label{fig:domain_high_low}
\end{figure}
We denote $\qOmegas_\qj:=\qOmegas\cap\qOmega_\qj$ and $\qOmegap_\qj:=\qOmegap\cap\qOmega_\qj$.
Now, for each $\qi=1,\ldots,\qL_\qj$ let $\qOmegas_{\qj,\qi}$ denote the $\qi$-th connected component of $\qOmegas_\qj$, where $\qL_\qj$ denotes the total number of connected components of $\qOmegas_\qj$. If $\qOmegas_\qj=\emptyset$, we set $\qOmegas_{\qj,1}=\qOmega_\qj$ and $\qL_\qj=1$.

Now, we define the following subspace of $\qVscr(\qOmega_\qj)$:
\begin{equation}
\label{eq:comp_space}
\qVscrcSE(\qOmega_\qj):=\left\{\qphi\in\qVscr(\qOmega_\qj)\,|\, ~~\dint{\qOmegas_{\qj,\qi}}{}\qphi\qdbx=0\ \text{, for }\qi=1,\ldots,\qL_\qj\right\}.
\end{equation}
It is straightforward to see that any $\qL_\qj +1$-dimensional subspace of $\qVscr(\qOmega_\qj)$ has a non-trivial intersection with $\qVscrcSE(\qOmega_\qj)$.
Thus, by \eqref{eq:min_max} there exists a non-zero $\qphi\in\qVscrcSE(\qOmega_\qj)$ such that
\begin{equation}
\label{eq:est_lambda}
\qlambda^\qj_{\qL_\qj+1}\ge \dfrac{\qaSE{\qOmega_\qj}{\qphi}{\qphi}}{\qmsumSE{\qOmega_\qj}{\qphi}{\qphi}}.
\end{equation}
Using the definitions of $\qaSE{\qOmega_\qj}{\cdot}{\cdot}$ and $\qmsumSE{\qOmega_\qj}{\cdot}{\cdot}$ we note that
\begin{equation}
\label{eq:est_m}
\begin{array}{rcl}
\qmsumSE{\qOmega_\qj}{\qphi}{\qphi}
&=& \dsum{\qi\in\qI_\qj}{} \qaSE{\qOmega_\qj}{\qxi_\qj\qxi_\qi \qphi}{\qxi_\qj\qxi_\qi \qphi}\\
&=& \dsum{\qi\in\qI_\qj}{} \dint{\qOmega_\qj}{} \qkappa |\grad(\qxi_\qj\qxi_\qi \qphi)|^2 \qdbx\\
&\le& 2  \dsum{\qi\in\qI_\qj}{} \dint{\qOmega_\qj}{} \( \qkappa | (\grad(\qxi_\qj\qxi_\qi) \qphi |^2 + |\qxi_\qj\qxi_\qi \grad\qphi|^2 \) \qdbx \\
&\le& \qC \dsum{\qi\in\qI_\qj}{} \dint{\qOmega_\qj}{} \qkappa \( \( \qH^{-1} \qphi\)^2 + |\grad\qphi|^2\)\qdbx \\
&\le& \qC \qnI \( \qH^{-2} \dint{\qOmega_\qj}{} \qkappa \qphi^2 \qdbx + \qaSE{\qOmega_\qj}{\qphi}{\qphi}\), \\
\end{array}
\end{equation}
where $\qC$ only depends on the choice of the partition of unity.
Furthermore, we observe that
$$
\label{eq:est_eig_2}
\begin{array}{rcl}
\dint{\qOmega_\qj}{}\qkappa \qphi^2 \qdbx
&=& \dint{\qOmegap_\qj}{}\qphi^2\qdbx + \qkappamax\dsum{\qi=1}{\qL_\qj}\dint{\qOmegas_{\qj,\qi}}{} \qphi^2 \qdbx \\
&\le& \qC \qH^2 \( \dint{\qOmegap_\qj}{}|\grad\qphi|^2\qdbx + \qkappamax\dsum{\qi=1}{\qL_\qj}\dint{\qOmegas_{\qj,\qi}}{} |\grad\qphi|^2 \qdbx \) \\
&=& \qC \qH^2 \dint{\qOmega_\qj}{} \qkappa |\grad\qphi|^2\qdbx,
\end{array}
$$
where we have used Poincar\'e's inequality, which is possible since $\qphi\in\qVscrcSE(\qOmega_\qj)$. Here, $\qC$ is a constant which only depends on the geometries of $\qOmegap_\qj$ and $\qOmegas_{\qj,\qi}$, $\qi=1,\ldots,\qL_\qj$. Thus, we obtain
$$
\qmsumSE{\qOmega_\qj}{\qphi}{\qphi} \le \qC \qaSE{\qOmega_\qj}{\qphi}{\qphi},
$$
which together with \eqref{eq:est_lambda} yields a uniform (with respect to $\qkappamax/\qkappamin$ and $\qH$) lower bound for $\qlambda^\qj_{\qL_\qj+1}$. Thus, we have verified assumption (A5) with $\qtaulambda$ independent of $\qkappamax/\qkappamin$ and $\qH$. In particular we see that for a suitably chosen $\qtaulambda$ the number of generalized eigenvalues below $\qtaulambda^{-1}$ and satisfying \eqref{eq:gen_eigenval} is bounded from above by the number of connected components in $\qOmegas_\qj$, i.e., $\qL_\qj$.

\subsection{The Scalar Elliptic Case -- Mixed Formulation}
\label{sec:Darcy}

In this Subsection we consider the mixed formulation of the scalar elliptic equation, also known as Darcy's equations, in 2 spatial dimensions, i.e., $\qn=2$, modeling flow in porous media
\begin{equation}
\label{eq:Darcy}
\left\{
\begin{array}{rcll}
\grad\qp+\qmu\qkappa^{-1}\qbu &=& \qbf&\text{in } \qOmega,\\
\div\qbu&=&0&\text{in }\qOmega,\\
\qbu\cdot\qbn&=&0&\text{on }\qdOmega.
\end{array}
\right. 
\end{equation}
Here, $\qp\in L^2_0(\qOmega):=L^2(\qOmega)/\qbR$ denotes the pressure, $\qbu\in H(div;\, \qOmega):=\{\qbv\in (L^2(\qOmega))^2\, |\, \div \qbv\in L^2(\qOmega)\}$ is the velocity, $\qbf\in (L^2(\qOmega))^2$ is a forcing term, and $\qbn$ denotes the unit outer normal vector to $\qdOmega$. The viscosity $\qmu$ is a positive constant, and $\qkappa\in L^\infty(\qOmega)$ is a positive function.
With $H_0(div;\, \qOmega):= \{\qbv\in H(div;\, \qOmega)\,|\, \qbv\cdot\qbn=0 \text{ on } \qdOmega\}$ the variational formulation of Darcy's problem is given by: Find $(\qbu,\, \qp)\in (H_0(div;\, \qOmega),\, L^2_0(\qOmega))$ such that for all $(\qbv,\, \qq)\in (H_0(div;\, \qOmega),\, L^2_0(\qOmega))$ we have
\begin{equation}
\label{eq:Darcy_primal}
\dint{\qOmega}{}\qmu \qkappainv \qbu \cdot \qbv \qdbx - \dint{\qOmega}{}\qp \div \qbv \qdbx - \dint{\qOmega}{}\qq \div \qbu \qdbx = \dint{\qOmega}{}\qbf \cdot \qbv \qdbx.
\end{equation}
It is well-known (see e.g.\ \cite[Chapter 12, p.\ 300]{BS02}) that problem \eqref{eq:Darcy_primal} is equivalent to the following problem: find $\qbu$ in the subspace  
$ H_{0}(div_0; \, \qOmega):=\{ \qbv \in H_0(div;\,\qOmega)\, |\, \div \qbv \equiv 0\}$ such that
$$
\dint{\qOmega}{}\qmu \qkappainv \qbu \cdot \qbv \qdbx = \dint{\qOmega}{}\qbf \cdot \qbv \qdbx
\quad \forall v \in H_{0}(div_0;\, \qOmega)
$$

Let us additionally assume that $\qOmega$ is simply connected. Then, we have (see e.g.\ \cite{Girault-R-1979}) that there exists exactly one $ \qphi \in H^1_0(\qOmega)$ such that $\curl \qphi=\qbu$, where $\curl \qphi:=\[\frac{\partial \qphi}{\partial \qx_2},~ -\frac{\partial \qphi}{\partial \qx_1} \]$.
This leads to the variational form of Darcy's problem in stream function formulation: Find $\qphi\in \qVscr_0:= H^1_0(\qOmega)$ such that for all $\qpsi\in \qVscr_0$ we have
\begin{equation}
\label{eq:Darcy_stream}
\qaD{}{\qphi}{\qpsi}:=\dint{\qOmega}{} \qmu \qkappainv\curl \qphi\cdot\curl \qpsi \qdbx = \dint{\qOmega}{}\qbf \cdot (\curl \qpsi) \qdbx.
\end{equation}
Note that $\qaD{}{\qphi}{\qpsi}=\dint{\qOmega}{} \qmu \qkappainv\grad \qphi\cdot\grad \qpsi \qdbx$. Thus, for $\qVscr_0(\qOmega_\qj)$ and $\qxi_\qj$ as chosen in Subsection~\ref{sec:ScalarElliptic} for $\qj=1,\ldots,\qnbx$ we can readily verify assumptions (A1)--(A4). Let us assume that, as above, $\qkappa$ only assumes two values. Then, we can perform exactly the same argument as in Subsection~\ref{sec:ScalarElliptic} (with $\qkappa$ replaced by $\qmu\qkappainv$) to establish the validity of (A5).
This in turn establishes the robustness (with respect to $\qkappamax/\qkappamin$ and $\qH$) of the stable decomposition \eqref{eq:stab_decom_specific} corresponding to the stream function formulation with its bilinear form $\qaD{}{\cdot}{\cdot}$. Thus, we may robustly precondition \eqref{eq:Darcy_stream}, namely, solve for $\qphi$ and recover $\qbu=\curl \qphi$ from \eqref{eq:Darcy_primal}.

An equivalent approach, that we use in Section~\ref{sec:numerics} to compute a solution of \eqref{eq:Darcy_primal}, is somewhat different and outlined in the following Remark.

\begin{remark}
\label{rem:stab_decom_mixed}
Instead of solving the stream function formulation for $\qphi$ and then recovering $\qbu=\curl \qphi$, one may equivalently use the coarse space corresponding to \eqref{eq:Darcy_stream} to construct a coarse space corresponding to \eqref{eq:Darcy_primal} by applying $\curl$ to the coarse stream basis functions. This then yields an equivalent robust additive Schwarz preconditioner for \eqref{eq:Darcy_primal} (for details see \cite[Section 10.4.2]{Mat08}).
\end{remark}

\subsection{Stokes' Equation}
\label{sec:Stokes}

As for the mixed form of the elliptic equation we assume that $\qOmega\subset \qbR^2$ is simply connected. Then we consider Stokes' equations modeling slow viscous flows
$$
\left\{
\begin{array}{rcll}
-\qmu\laplace\qbu+\grad\qp &=& \qbf&\text{in } \qOmega,\\
\div\qbu&=&0&\text{in }\qOmega,\\
\qbu&=&\bm{0}&\text{on }\qdOmega,
\end{array}
\right.
$$
where $\qp\in L_0^2(\qOmega)$, $\qbu\in (H^1_0(\qOmega))^2$, $\qbf\in (L^2(\qOmega))^2$, and $\qmu\in \qbR^+$. The variational formulation of the Stokes problem is: Find $(\qbu,\, \qp) \in ( (H^1_0(\qOmega))^2,\, L^2_0(\qOmega) )$ such that for all $(\qbv,\, \qq) \in ( (H^1_0(\qOmega))^2,\, L^2_0(\qOmega) )$ we have
\begin{equation}
\label{eq:Stokes_primal}
\dint{\qOmega}{}\qmu \grad \qbu : \grad \qbv \qdbx - \dint{\qOmega}{}\qp \div \qbv\qdbx -\dint{\qOmega}{}\qq \div\qbu \qdbx = \dint{\qOmega}{}\qbf \cdot \qbv \qdbx,
\end{equation}
where $\grad \qbu : \grad \qbv := \dsum{\qi,\qj=1}{2} \dfrac{\partial \qu_\qi}{\partial \qx_j} \dfrac{\partial \qv_\qi}{\partial \qx_j}$ denotes the usual Frobenius product.

Analogously to Section~\ref{sec:Darcy} we can formulate an equivalent problem for stream functions: Find $\qphi \in \qVscr_0:=\left\{\qpsi\in H^2(\qOmega)\cap H^1_0(\qOmega)\, |\, \frac{\partial \qpsi}{\partial \qbn}|_{\partial\qOmega}=0\right\}$ such that for all $\qphi \in \qVscr_0$ 
\begin{equation}
\label{eq:Stokes_stream}
\qaS{}{\qphi}{\qpsi}
:=\dint{\qOmega}{}\qmu\grad(\curl \qphi):\grad(\curl \qpsi) \qdbx = \dint{\qOmega}{}\qbf \cdot \curl \qpsi \qdbx.
\end{equation}
For a sufficiently regular partition of unity $\{\qxi_\qj\}_{\qj=1}^{\qnbx}$ and spaces $\qVscr_0(\qOmega_\qj)$, $\qj=1,\ldots,\qnbx$, defined as $\qVscr_0(\qOmega_\qj):=\left\{\qpsi\in H^2(\qOmega_\qj)\cap H^1_0(\qOmega_\qj)\, |\, \frac{\partial \qpsi}{\partial \qbn}|_{\partial\qOmega_\qj}=0\right\}$, we can readily verify (A1)--(A4). For the verification of (A5), we define for $\qj=1,\ldots,\qnbx$
$$
\qVscrcS(\qOmega_\qj):=\left\{\qphi\in\qVscr(\qOmega_\qj)\, | \, \dint{\qOmega_\qj}{} \qphi=0,
\, \, \dint{\qOmega_\qj}{} \grad \qphi\qdbx=\bm{0}\right\}.
$$
As above, it is straightforward to see that any $4$-dimensional subspace has a non-empty intersection with $\qVscrcS(\qOmega_\qj)$. Thus, by again using the min-max principle we see that there exists a $\qphi\in \qVscrcS(\qOmega_\qj)$ such that
\begin{equation}
\label{eq:eig_Stokes}
\qlambda^\qj_4\ge \dfrac{\qaS{\qOmega_\qj}{\qphi}{\qphi}}{\qmsumS{\qOmega_\qj}{\qphi}{\qphi}}.
\end{equation}
Using the definitions of $\qaS{\qOmega_\qj}{\cdot}{\cdot}$ and $\qmsumS{\qOmega_\qj}{\cdot}{\cdot}$ (see \eqref{eq:def_msum} with $\qa{}{\cdot}{\cdot}$ replaced by $\qaS{}{\cdot}{\cdot}$) we note that
$$
\begin{array}{rcl}
\qmsumS{\qOmega_\qj}{\qphi}{\qphi}
&=& \dsum{\qi\in\qI_\qj}{} \qaS{\qOmega_\qj}{\qxi_\qj\qxi_\qi \qphi}{\qxi_\qj\qxi_\qi \qphi}\\
&=& \dsum{\qi\in\qI_\qj}{} \dint{\qOmega_\qj}{} \qmu \Hess{\qxi_\qj\qxi_\qi\qphi}:\Hess{\qxi_\qj\qxi_\qi\qphi} \qdbx\\
&=& \dsum{\qi\in\qI_\qj}{} \dint{\qOmega_\qj}{} \qmu \dnorm{\qphi\Hess{\qxi_\qi\qxi_\qj} + \qxi_\qi\qxi_\qj \Hess{\qphi} + \grad(\qxi_\qi\qxi_\qj) \otimes \grad\qphi + \grad\qphi \otimes \grad(\qxi_\qi\qxi_\qj)}{\qFcal}^2 \qdbx \\
&\le& \qC \qnI \dint{\qOmega_\qj}{} \qmu (\qphi^2 \qH^{-4} + \Hess{\qphi}:\Hess{\qphi} + \qH^{-2} (\grad\qphi)^2) \qdbx,
\end{array}
$$
where $\qC$ only depends on the particular choice of the partition of unity. Here $\Hess{\qphi}:= \[\dfrac{\partial^2 \qphi}{\partial \qx_\qi \partial \qx_\qj}\]_{\qi,\qj=1,\, 2}$ denotes the Hessian of $ \qphi$, $\otimes$ the tensor product, and $\dnorm{\cdot}{\qFcal}$ denotes the Frobenius norm associated with the Frobenius product defined above. Since $\qphi\in\qVscrcS(\qOmega_\qj)$ we may apply Poincar\'e's inequality to $\qphi$ and its first derivatives. Thus, we obtain
\begin{equation}
\label{eq:est_eig_S1}
\qmsumS{\qOmega_\qj}{\qphi}{\qphi}
\le \qC \qnI \dint{\qOmega_\qj}{} \qmu \Hess{\qphi}:\Hess{\qphi} \qdbx = \qC \qaS{\qOmega_\qj}{\qphi}{\qphi},
\end{equation}
where $\qC$ only depends on the partition of unity and the shape of $\qOmega_\qj$ but is independent of $\qH$. Combining \eqref{eq:est_eig_S1} with \eqref{eq:eig_Stokes} we obtain
$
\qlambda^\qj_4\ge \qC
$
with $\qC$ independent of $\qH$. This verifies (A5) and thus establishes that the decomposition  \eqref{eq:stab_decom_specific} and the corresponding additive Schwarz preconditioner are robust with respect to $\qH$. As outlined in Remark~\ref{rem:stab_decom_mixed} we can also obtain a robust additive Schwarz preconditioner for \eqref{eq:Stokes_primal}.

\subsection{Brinkman's Equation}
\label{sec:Brinkman}

As for Darcy's and Stokes' problem, we assume that $\qOmega\subset \qbR^2$ is simply connected. Brinkman's problem modeling flows in highly porous media is given by (cf.\ \cite{Bri47_1})
\begin{equation}
\label{eq:Brinkman}
\left\{
\begin{array}{rcll}
-\qmu\laplace\qbu+\grad\qp+\qmu\qkappa^{-1}\qbu &=& \qbf&\text{in } \qOmega,\\
\div\qbu&=&0&\text{in }\qOmega,\\
\qbu&=&\bm{0}&\text{on }\qdOmega,
\end{array}
\right.
\end{equation}
where $\qp$, $\qbu$, $\qbf$, and $\qmu$ are chosen as in the Stokes' case and $\qkappa$ as in the Darcy case.
The variational formulation of the Brinkman problem is: Find $(\qbu,\, \qp) \in ( (H^1_0(\qOmega))^2,\, L^2_0(\qOmega) )$ such that for all $(\qbv,\, \qq) \in ( (H^1_0(\qOmega))^2,\, L^2_0(\qOmega) )$ we have
\begin{equation}
\label{eq:Brinkman_primal}
\dint{\qOmega}{}\qmu \grad \qbu : \grad \qbv \qdbx + \dint{\qOmega}{}\qmu \qkappainv \qbu \cdot \qbv \qdbx - \dint{\qOmega}{}\qp \div \qbv\qdbx -\dint{\qOmega}{}\qq \div\qbu \qdbx = \dint{\qOmega}{}\qbf \cdot \qbv \qdbx.
\end{equation}

Again, we adopt the setting of stream functions. For $\qVscr_0$ as in section~\eqref{sec:Stokes}, the variational stream function formulation reads: Find $\qphi \in \qVscr_0$ such that for all $\qpsi\in \qVscr_0$ we have
\begin{equation}
\label{eq:Brinkman_stream}
\qaB{}{\qphi}{\qpsi}:=\dint{\qOmega}{}\qmu \(\grad(\curl \qphi):\grad(\curl \qpsi) +\qkappainv \curl\qphi \cdot \curl\qpsi\) \qdbx = \dint{\qOmega}{}\qbf \cdot \curl \qpsi \qdbx.
\end{equation}
With $\qxi_\qj$ and $\qVscr_0(\qOmega_\qj)$ as in section~\ref{sec:Stokes} for $\qj=1,\ldots,\qnbx$ we readily verify (A1)--(A4).

Note that
$$
\qaB{}{\qphi}{\qpsi}= \qaS{}{\qphi}{\qpsi} + \qaD{}{\qphi}{\qpsi}\quad \text{ and }\quad \qmsumB{}{\qphi}{\qpsi} = \qmsumS{}{\qphi}{\qpsi} + \qmsumD{}{\qphi}{\qpsi},
$$
where $\qmsumB{}{\cdot}{\cdot}$ is defined according to \eqref{eq:def_msum} with $\qa{}{\cdot}{\cdot}$ replaced by $\qaB{}{\cdot}{\cdot}$.

Since for any $\qpsi\in\qVscr(\qOmega_\qj)$, $\qj=1,\ldots,\qnbx$,  we have \sloppy$\qmsumS{\qOmega_\qj}{\qpsi}{\qpsi}, \ \qmsumD{\qOmega_\qj}{\qpsi}{\qpsi}>0$ and \sloppy$\qaS{\qOmega_\qj}{\qpsi}{\qpsi},\ \qaD{\qOmega_\qj}{\qpsi}{\qpsi} \ge0$ we have
\begin{equation}
\label{eq:est_frac_split}
\dfrac{\qaB{\qOmega_\qj}{\qpsi}{\qpsi}}{\qmsumB{\qOmega_\qj}{\qpsi}{\qpsi}}
=
\dfrac{\qaS{\qOmega_\qj}{\qpsi}{\qpsi} + \qaD{\qOmega_\qj}{\qpsi}{\qpsi}}{\qmsumS{\qOmega_\qj}{\qpsi}{\qpsi} + \qmsumD{\qOmega_\qj}{\qpsi}{\qpsi}}
\ge
\min\left\{ \dfrac{\qaS{\qOmega_\qj}{\qpsi}{\qpsi}}{\qmsumS{\qOmega_\qj}{\qpsi}{\qpsi}},\ \dfrac{\qaD{\qOmega_\qj}{\qpsi}{\qpsi}}{\qmsumD{\qOmega_\qj}{\qpsi}{\qpsi}}\right\}.
\end{equation}
This is an immediate consequence of the following inequality, valid for $\qbeta_1,\qbeta_2 \ge 0$ and $\qbeta_3,\ \qbeta_4 >0$ 
$$
\dfrac{\qbeta_1 + \qbeta_2}{\qbeta_3 + \qbeta_4}
=
\dfrac{
\dfrac{\qbeta_1}{\qbeta_3\qbeta_4} + \dfrac{\qbeta_2}{\qbeta_3\qbeta_4}
}
{
\dfrac{1}{\qbeta_4} + \dfrac{1}{\qbeta_3}
}
\ge
\left\{
\begin{array}{ll}
\dfrac{\qbeta_2}{\qbeta_4}
, & \text{ if } \dfrac{\qbeta_1}{\qbeta_3}\ge \dfrac{\qbeta_2}{\qbeta_4}\\
\dfrac{\qbeta_1}{\qbeta_3}
, & \text{ if } \dfrac{\qbeta_2}{\qbeta_4} \ge \dfrac{\qbeta_1}{\qbeta_3}\\
\end{array}
\right\}
\ge
\min\left\{ \dfrac{\qbeta_1}{\qbeta_3},\ \dfrac{\qbeta_2}{\qbeta_4} \right\}. 
$$
Combining \eqref{eq:est_frac_split} with the results from Subsections~\ref{sec:Darcy} and \ref{sec:Stokes} we obtain
$$
\qlambda^\qj_{\max\{\qL_{\qj}+1,~4\}} \ge \qC,
$$
where $\qC$ is independent of $\qH$ and $\qkappamax/\qkappamin$, and where $\qL_\qj$ is chosen as in Subsection~\ref{sec:Darcy}. This verifies (A5) and thus establishes that the decomposition introduced by \eqref{eq:stab_decom_specific} and the corresponding additive Schwarz preconditioner are robust with respect to $\qH$ and $\qkappamax/\qkappamin$. As for the Darcy and the Stokes case, we can also obtain an equivalent robust additive Schwarz preconditioner for \eqref{eq:Brinkman_primal} (see Remark~\ref{rem:stab_decom_mixed}).


\section{Reducing the Dimension of the Coarse Space}
\label{sec:dim_coarse}
In the exposition above, no assumptions except (A4) and (implicitly) (A5) were made about the choice of the partition of unity $\{\qxi_\qj\}_{\qj=1}^{\qnbx}$. In this section, we investigate possibilities of making a particular choice of the partition of unity denoted by $\{\qxit_\qj\}_{\qj=1}^{\qnbx}$ that results in a reduction of the dimension of the coarse space $\qVscr_\qH$. The idea is that by replacing $\{\qxi_\qj\}_{\qj=1}^{\qnbx}$ with $\{\qxit_\qj\}_{\qj=1}^{\qnbx}$  one can avoid the asymptotically small eigenvalues for those connected components of $\qOmegas$ which do not touch the boundary of any coarse cell $\qTH\in\qTcalH$. For this, we again assume the scalar elliptic setting as in Section~\ref{sec:ScalarElliptic}.

Let $\{\qxi_\qj\}_{\qj=1}^{\qnbx}$ be a standard partition of unity as above. For each $\qj=1,\ldots,\qnbx$ and each $\qTH\subset \qOmega_\qj$, let $\qxit_\qj|_{\qTH}$ be a solution of
$$
-\div(\qkappa \grad \qxit_\qj) = 0,\quad \text{in } \qTH, \qquad \qxit_\qj = \qxi_\qj, \quad \text{on } \partial \qTH.
$$
The corresponding variational formulation reads: Find $\qxit_\qj|_{\qTH}\in H^1_0(\qTH)+\qxi_\qj$ such that for all $\qpsi\in H^1_0(\qTH)$ we have
\begin{equation}
\label{eq:xi_multiscale}
\qaSE{\qTH}{\qxit_\qj}{\qpsi}=0,\quad \forall\, \qpsi\in H^1_0(\qTH).
\end{equation}
We set $\qxit_\qj\equiv 0$ in $\qOmega\backslash\qOmega_\qj$ and check whether with $\{\qxit_\qj\}_{\qj=1}^{\qnbx}$ instead of $\{\qxi_\qj\}_{\qj=1}^{\qnbx}$ (A4) and (A5) are satisfied.

As in Subsection~\ref{sec:ScalarElliptic} let $\qOmegas_{\qj,\qk}$, $\qk=1,\ldots,\qL_\qj$ be the $\qk$-th connected component of $\qOmegas_\qj$. Without loss of generality we may assume a numbering such that $\qOmegas_{\qj,\qk}$ for $\qk=1,\ldots,\qLt_\qj (\le \qL_\qj)$ is a connected component for which $\qOmegasc_{\qj,\qk}\cap(\partial\qTH\backslash \qdOmega_\qj)\ne \emptyset$ for some $\qTH\subset\qOmega_\qj$. Note, that in general $\qLt_\qj\le\qL_\qj$ and that $\qL_\qj -\qLt_\qj$ is precisely the number of connected components of $\qOmegas_\qj$ which do not touch an edge of $\qTH\subset\qOmega_\qj$ interior to $\qOmega_\qj$, i.e., for $\qi=\qLt_\qj+1,\ldots, \qL_\qj$ we have that $\qOmegasc_{\qj,\qi}\cap (\partial\qTH\backslash \qdOmega_\qj) = \emptyset$.
\begin{figure}
\resizebox{.5\textwidth}{!}{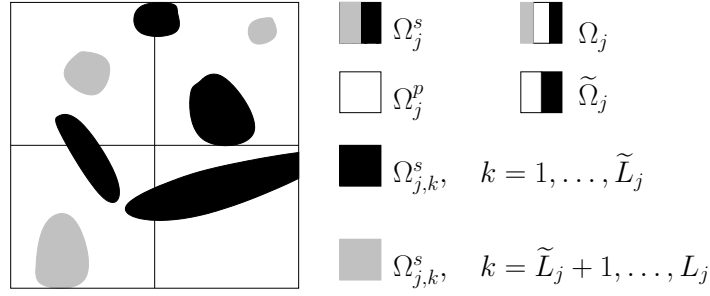}
\caption{Subdomain with connected components of $\qOmegas_\qj$. The black connected components are those touching an edge interior to $\qOmega_\qj$. The grey connected components are the remaining ones. In this configuration $\qLt_\qj=4$ and $\qL_\qj=7$.}
\label{fig:domain_high_low_gr}
\end{figure} 
We also define $\qOmegat_\qj:=\qOmega_\qj \backslash (\dcup{\qk=\qLt_\qj+1}{\qL_\qj}\qOmegasc_{\qj,\qk} )$. These notations are illustrated in Figure~\ref{fig:domain_high_low_gr}.

To proceed we make the following assumption:
\begin{itemize}
\item[\qAft] For
$\qxit_\qj$ defined as above, we have 
$$\dnorm{\grad \qxit_\qj}{L^\infty(\qOmega_\qj)}\le \qC \qH^{-1}~~\mbox{ and } ~~ \dnorm{\qkappamax \grad\qxit_\qj}{L^\infty(\qOmega_\qj \backslash \qOmegat_\qj)} \le \qC \qH^{-1},\qj=1,\ldots,\qnbx,$$  
where $\qC$ is independent of $\qkappamax/\qkappamin$ and $\qH$.
\end{itemize}
Note that by \cite[Lemma~3.1]{EILRW08} $\dnorm{\grad \qxit_\qj}{L^2(\qOmega_\qj)}^2\le \qC \qH^{\qn-2}$ and that by essentially the same argument as in \cite[Theorem 4.3 and 4.5]{IvanRob} we have that $\dnorm{\qkappamax \grad\qxit_\qj}{L^2(\qOmega_\qj\backslash \qOmegat_\qj)}^{2} \le \qC \qH^{\qn-2}$, where as above $\qn$ denotes the spatial dimension. To the best of our knowledge, obtaining rigorous $L^\infty$-estimates as stated in \qAft\ is still an open problem and beyond the scope of this paper. Nevertheless, we consider reducing the verification of (A4) and (A5) to establishing \qAft\ a significant improvement.

It is easy to see that $\dsum{\qj=1}{\qnbx}\qxit_\qj\equiv 1$ and that $\supp(\qxit_\qj)=\qOmegac_\qj$ for any $\qj=1,\ldots,\qnbx$. Thus, to establish the validity of (A4) it remains to verify that $\qxit_\qj\qpsi\in\qVscr_0=H^1_0(\qOmega)$ and $(\qxit_\qj \qpsi)|_{\qOmega_\qj}\in \qVscr_0(\qOmega_\qj)$ for all $\qpsi\in \qVscr_0$. For this we restrict to the case of two spatial dimensions, i.e., $\qn=2$:

Note that for some $\qepsilon > 0$ we have that $\qxit_\qj\in H^{1+\qepsilon}(\qOmega_\qj)$ (cf.\ \cite{Grisvard85}). Thus, by \cite[Theorem 7.57]{Ada78} we know that $\qxit_\qj \in L^\infty(\qOmega_\qj)$. Using this and \qAft\ we see that
$$
\begin{array}{rl}
\dnorm{\grad (\qxit_\qj\qpsi)}{L^2(\qOmega_\qj)} & \le \dnorm{\qxit_\qj \grad \qpsi}{L^2(\qOmega_\qj)} + \dnorm{\qpsi \grad \qxit_\qj}{L^2(\qOmega_\qj)} \\[2ex]
 & \le \qC \(  \dnorm{\grad \qpsi}{L^2(\qOmega_\qj)} + \qH^{-1}\dnorm{\qpsi}{L^2(\qOmega_\qj)} \) < \infty.
\end{array}
$$
It is furthermore easy to see that $\supp(\qxit_\qj\qpsi) \subset \qOmegac_\qj$, which establishes (A4).

Similarly to section~\ref{sec:ScalarElliptic} we now define
\begin{equation}
\label{eq:VscrtcSE}
\qVscrtcSE(\qOmega_\qj):=\left\{\qphi\in\qVscr(\qOmega_\qj)\,|\, \dint{\qOmegas_{\qj,\qk}}{}\qphi\qdbx=0\ \text{, for }\qk=1,\ldots,\qLt_\qj\right\},
\end{equation}
and by the min-max principle we know that there exists a $\qphi\in\qVscrtcSE(\qOmega_\qj)$ such that
\begin{equation}
\label{eq:est_lambda2}
\qlambdat^\qj_{\qLt_\qj+1}\ge \dfrac{\qaSE{\qOmega_\qj}{\qphi}{\qphi}}{\qmsumtSE{\qOmega_\qj}{\qphi}{\qphi}},
\end{equation}
where $\qmsumtSE{\qOmega_\qj}{\cdot}{\cdot}$ is defined as $\qmsumSE{\qOmega_\qj}{\cdot}{\cdot}$ with $\{\qxi_\qj\}_{\qj=1}^{\qnbx}$ replaced by $\{\qxit_\qj\}_{\qj=1}^{\qnbx}$.

In order to obtain a uniform (with respect to $\qkappamax/\qkappamin$ and $\qH$) lower bound for $\qlambdat^\qj_{\qLt_\qj+1}$ and thus establish (A5), we need to verify that
\begin{equation}
\label{eq:est_eig_mult}
\qmsumtSE{\qOmega_\qj}{\qphi}{\qphi} \le \qC \qaSE{\qOmega_\qj}{\qphi}{\qphi},
\end{equation}
with $\qC$ independent of $\qkappamax/\qkappamin$ and $\qH$.

Revisiting estimate \eqref{eq:est_m}, we obtain
$$
\qmsumtSE{\qOmega_\qj}{\qphi}{\qphi}
\le 2  \dsum{\qi\in\qI_\qj}{} \dint{\qOmega_\qj}{} \qkappa |\grad(\qxit_\qj\qxit_\qi) \qphi|^2 \qdbx + 2\qaSE{\qOmega_\qj}{\qphi}{\qphi}.
$$

Thus, it suffices to bound for any $\qi\in \qI_\qj$
$$
\dint{\qOmega_\qj}{} \qkappa |\grad(\qxit_\qj\qxit_\qi) \qphi|^2 \qdbx = \underbrace{\dint{\qOmegat_\qj}{}\qkappa |\grad(\qxit_\qj\qxit_\qi) \qphi|^2 \qdbx}_{=:\qE_4} + \underbrace{\dint{\qOmega_\qj \backslash \qOmegat_\qj}{}\qkappa |\grad(\qxit_\qj\qxit_\qi) \qphi|^2 \qdbx}_{=:\qE_5} 
$$
by $\qaSE{\qOmega_\qj}{\qphi}{\qphi}$.
To avoid unnecessary technicalities, we make the simplifying assumption that each connected component of $\qOmegat_\qj$ contains at least one $\qOmegas_{\qj,\qk}$ with $\qk=1,\ldots,\qLt_\qj$. If this assumption is violated, one simply needs to introduce additional conditions in \eqref{eq:VscrtcSE} ensuring that the average of functions is zero over each connected component of $\qOmegat_\qj$ that does not contain any $\qOmegas_{\qj,\qk}$ with $\qk=1,\ldots,\qLt_\qj$.

Assuming \qAft, we can find the required estimate of $\qE_4$ as follows:
Proceeding as in Subsection~\ref{sec:ScalarElliptic} we see
$$
\qE_4
\le \qC\,\, \qH^{-2} \dint{\qOmegat_\qj}{} \qkappa \qphi^2 \qdbx
\le \qC\,\, \dint{\qOmegat_\qj}{}\qkappa |\grad \qphi |^2 \qdbx 
\le \qC\,\, \qaSE{\qOmega_\qj}{\qphi}{\qphi},
$$
where we have used Poincar\'e's inequality, which is possible since $\qphi\in\qVscrtcSE(\qOmega_\qj)$. Due to \qAft, $\qC$ can be chosen independently of $\qkappamax/\qkappamin$ and $\qH$, but it may depend on the geometry of $\qOmegat_\qj$.

For an estimate of $\qE_5$ note that by \qAft\
$$
\begin{array}{rcl}
\qE_5 & \le & 2 \qkappamax \dint{\qOmega_{\qj}\backslash \qOmegat_{\qj}}{} \( |\grad \qxit_\qj|^2 + 
|\grad \qxit_\qi|^2 \) \qphi^2 \qdbx\\
& \le & \qC \qH^2 \dint{\qOmega_{\qj}\backslash \qOmegat_{\qj}}{}\qphi^2 \qdbx \le \qC \dint{\qOmega_{\qj}}{} |\grad\qphi|^2 \qdbx \le \qC \qaSE{\qOmega_\qj}{\qphi}{\qphi},
\end{array}
$$
where we have used Poincar\'e's inequality. This establishes \eqref{eq:est_eig_mult}, which yields the validity of (A5).

In analogy to \eqref{eq:coarse} we define the coarse space, called further multiscale spectral coarse space, that is constructed using $\{\qxit_\qj\}_{\qj=1}^{\qnbx}$ instead of $\{\qxi_\qj\}_{\qj=1}^{\qnbx}$ by
\begin{equation}
\label{eq:coarse_space_tilde}
\qVscrtH:=\qspan\{\qxit_\qj \widetilde \qvarphi^\qj_\qi\, |\, ~~\qj=1,\ldots,\qnbx \text{ and } \qi=1,\ldots,\qLt_\qj\},
\end{equation}
where $ \widetilde \qvarphi^\qj_\qi$ are given by \eqref{eq:gen_eigenval} with $\qmsum{\qOmega_j}{\cdot}{\cdot}$ replaced by $\qmsumtSE{\qOmega_j}{\cdot}{\cdot}$.

\begin{remark}
We note that for any subdomain $\qOmega_\qj$ with $\qOmegac_\qj \cap \qdOmega = \emptyset$, we have that $(0, \bm{1}_{\qOmega_{\qj}})$ is an eigenpair of the generalized eigenvalue problem posed on $\qOmega_\qj$, where $\bm{1}_{\qOmega_{\qj}}$ denotes the constant $1$-function on $\qOmega_{\qj}$. Thus, all multiscale partition of unity functions $\qxit_\qj$ corresponding to subdomains $\qOmega_\qj$ with $\qOmegac_\qj \cap \qdOmega = \emptyset$ are basis functions of our coarse space $\qVscrtH$. This observation allows the interpretation of our method as a procedure that enriches a multiscale coarse space (given by the span of the multiscale partition of unity functions) by features that cannot be represented locally. These features are incorporated by those eigenfunctions corresponding to non-zero (but small) eigenvalues.
\end{remark}

\section{Numerical Results}
\label{sec:numerics}

\subsection{General setting}

In this section, we investigate the performance of the overlapping Schwarz domain decomposition method with coarse spaces discussed above when applied to some specific problems described in Section \ref{sec:applications}. We have implemented this preconditioner in C++ using the finite element library deal.ii (cf.\ \cite{BHK07}). Our goal is to experimentally establish the robustness of additive Schwarz preconditioners with respect to contrast $\qkappamax/\qkappamin$.
 The comparison is made  using some 
coarse spaces known in the literature and 
the coarse spaces introduced in this paper. 
Namely, we consider the following coarse spaces.
\begin{enumerate}
\item \textit{Standard coarse space}, 
denoted by $ \qVscrH^{st} := \qspan\{\qxi_\qj\,|\quad \qxi_\qj|_{\qdOmega_\qj}\equiv 0 \text{ for }\qj=1,\ldots,\qnbx\}$, of  
standard partition of unity functions corresponding to interior coarse mesh nodes, which were introduced above (also, e.g \cite{TW05});
\item \textit{Multiscale coarse space}, denoted by $ \qVscrH^{ms} := \qspan\{\qxit_\qj\,|\quad \qxit_\qj|_{\qdOmega_\qj}\equiv 0 \text{ for } \qj=1,\ldots,\qnbx\}$, of functions that over each  $\qOmega_j$ are solutions of problem \eqref{eq:xi_multiscale} and correspond to interior coarse mesh nodes (cf.\ \cite{IvanRob});
\item \textit{Spectral coarse space}, defined in \eqref{eq:coarse} as
$
\qVscrH:=\qspan\{\qxi_\qj\qvarphi^\qj_\qi\, |\, ~~\qj=1,\ldots,\qnbx \text{ and } \qi=1,\ldots,\qL_\qj\}
$
\item \textit{Multiscale spectral coarse space}, defined in \eqref{eq:coarse_space_tilde} (used only in the scalar elliptic case, see Section~\ref{sec:dim_coarse}).
\end{enumerate}

In all numerical examples (unless stated otherwise), the threshold for taking into account eigenpairs for the construction of the coarse space is chosen to be $0.5$, i.e., $1/\qtaulambda = 0.5$. Note that in the finite dimensional case (A5) is satisfied for any choice of the threshold $1/\qtaulambda$. However, in practice one is generally interested in choosing $1/\qtaulambda$ not too large to avoid an unnecessarily large dimension of the coarse space.

\begin{figure}[h!]
\includegraphics[width=.4\textwidth]{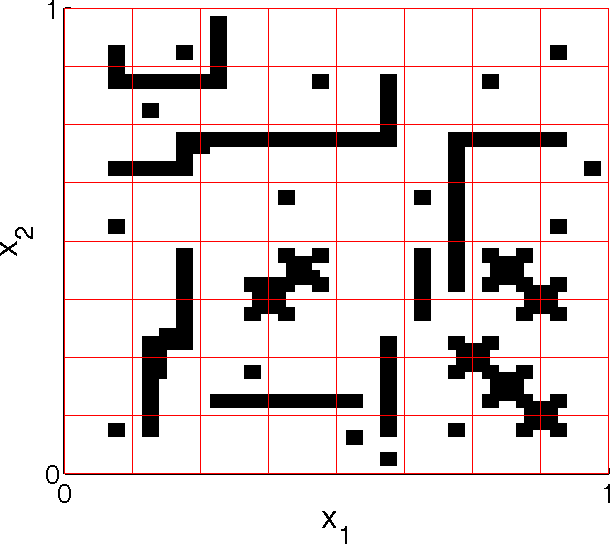} \hspace{1cm}
\includegraphics[width=.4\textwidth]{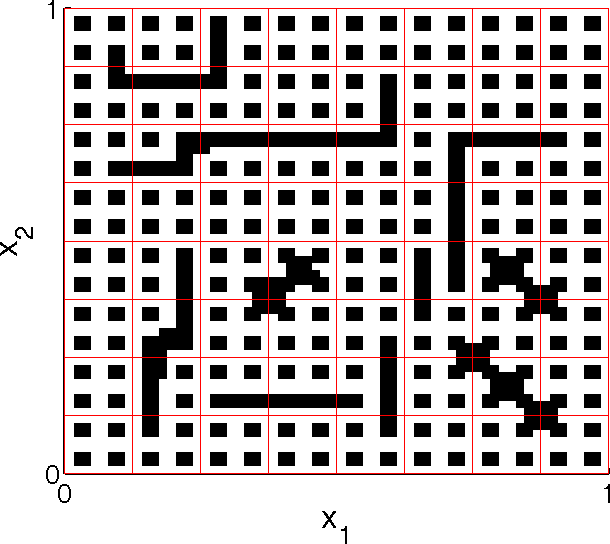}
\caption{Two sample geometries: Geometry 1 (left) and Geometry 2 (right); the regions of low (white) and high (black) values of the coefficients. The mesh indicates the coarse triangulation.}
\label{fig:geometry}
\end{figure}
We consider Geometries~$1-4$ (see Figures~\ref{fig:geometry} and \ref{fig:geom_256}), where for Geometries $1-3$
$\qkappa$ is equal to $\qkappamin$ and $\qkappamax$ in the white and black regions, respectively,
and for Geometry~$4$ $\qkappa$ is given as shown in the logarithmic plot of Figure~\ref{fig:geom_256}(b).

The geometries shown in Figure~\ref{fig:geometry} differ by the number of connected subregions with high permeability.  
The goal for these two different distributions of the high contrast is to (1) test the robustness of the developed preconditioners with respect to the contrast and (2) show the benefits of the multiscale coarse space in the case of a large number of not connected, isolated, inclusions with high conductivity.

In the abstract setting we replace the spaces $\qVscr $,  $\qVscr_0$, and $\qVscrH$  by their finite element counterparts. For this, we use the fine grid which is obtained from the coarse grid by subdividing the coarse grid elements into a number of finer elements. For 
Geometry~1 and 2, we make an initial $8 \times 8$ mesh and introduce in each rectangular element an 
$8 \times 8$ fine mesh, denoted by ${\mathcal T}_h$. Then, the spaces  $\qVscr_0$, $\qVscr_0(\qOmega_\qj)$ and  $\qVscr(\qOmega_\qj)$ are finite element spaces corresponding to this mesh for a specific finite element, which  needs to be chosen appropriately for the problem under investigation, e.g., Lagrange finite elements for the scalar elliptic problem in Galerkin formulation. In order not to overburden the notations, we have omitted the dependence of the spaces upon the fine-grid mesh size hoping that this will not lead to a confusion.

Unfortunately, such choice of the spaces does not satisfy assumption (A4), since generally the product $\xi_j v$ of a standard partition of unity function and a finite element function does not belong to the finite element space $\qVscr$. There are two ways to overcome this problem: (1) to project $\xi_j v$  back to the finite element space using the form $a(\cdot, \cdot)$ or (2) to use the finite element interpolant of $\xi_j v$ in the finite element space. The projection is a local operation, but
involves inverting some local stiffness matrices and could be unnecessarily expensive. The interpolation option, which was used in our computations, is straightforward to implement, but is not covered immediately by the abstract setting developed above. However, a perturbation analysis could show that this is a viable practical approach whose rigorous study is a subject of our future research.

\subsection{Numerical Experiments for Geometry 1}

In this subsection we use the standard coarse space $ \qVscrH^{st}$ and the spectral coarse space $\qVscrH$, defined in \eqref{eq:coarse}.

\subsubsection{Scalar Elliptic Problem -- Galerkin Formulation (see Section~\ref{sec:ScalarElliptic})}
Here, the finite element space is the space of Lagrange finite elements of degree $1$. The right hand side $\qf$ in \eqref{eq:scalar_elliptic} is chosen to compensate for the boundary condition of linear temperature drop in $\qx$-direction, i.e., $\qphi(\qbx)=1-\qx_1$ on $\qdOmega$.
The dimension of the fine-grid space is $4225$. In the PCG method we iterate to achieve a relative reduction of the preconditioned residual of $1e-6$.
In Tables \ref{tab:scalar_elliptic_simple} and \ref{tab:scalar_elliptic_var_thresh} we present the results of two kinds of numerical experiments on the problem described in Subsection~\ref{sec:ScalarElliptic} for Geometry~1 with contrast $\qkappamax/\qkappamin$ increasing from $1e2$ to $1e6$. As partition of unity $\{\qxi_\qj\}_{\qj=1}^{\qnbx}$, we use Lagrange finite element functions of degree 1 corresponding to the coarse mesh $\qTcalH$.

In Table~\ref{tab:scalar_elliptic_simple} we compare the number of PCG-iterations and the condition numbers for two preconditioners based on the standard coarse space $\qVscrH^{st}$ (consisting only of the coarse Lagrange finite element functions) and the spectral coarse spaces $\qVscrtH$ generated by our method, respectively. 
The standard coarse space has fixed dimension $49$. The method performs well for low contrasts, but the condition number of the preconditioned systems as well as the number of iterations grow with increasing contrast. The spectral coarse space keeps the condition number independent of the contrast, which is in agreement with our theory. 

It seems that the results in the number of iterations for the space $\qVscrH^{st}$ in the last row in Table~\ref{tab:scalar_elliptic_simple} deviates from the general trend. We note that for all cases we run the PCG-method with the same stopping criterion, i.e., reduction of the initial preconditioned residual by a factor of $1e-6$. However, in this case the condition number of the preconditioned system is $ 1.77e5$. Therefore, after reducing the initial preconditioned residual by a factor of $1e-6$ 
we may still be far away from the solution. 
Apparently, for larger condition numbers we 
may need many more iterations to compute the solution accurately.


\begin{table}[h!]
\begin{tabular}{V{2.5} c V{2.5} c| c| c V{2.5} c |  c|c V{2.5}}\clineB{2-7}{2.5}
\multicolumn{1}{ c V{2.5} }{} &  \multicolumn{3}{ c V{2.5}}{Standard coarse space $\qVscrH^{st}$} &  \multicolumn{3}{c V{2.5} }{Spectral coarse space $\qVscrH$} 
\\ \hlineB{2.5}
$\qkappamax/\qkappamin$ & \# iter. & dim $\qVscrH^{st}$ & cond.\ num.& \# iter. & dim. $\qVscrH$ & cond.\ num.\ \\ \hlineB{2.5}
$1e2$ & 29 &  49 &  2.29e1 & $25$  &  $76$ & $15.59$ \\ \hline
$1e3$ & 50 &  49 &  1.88e2 & $21$  & $145$ & $11.51$ \\ \hline
$1e4$ & 55 &  49 &  1.79e3 & $18$  & $162$ & $6.20$ \\ \hline
$1e5$ & 67 &  49 &  1.78e4 & $18$  & $162$ & $6.18$ \\ \hline
$1e6$ & 66 &  49 &  1.77e5 & $19$  & $162$ & $6.19$ \\ \hlineB{2.5}
\end{tabular}\\[2ex] 
\caption{Elliptic Problem of Second Order: Numerical results for standard and spectral coarse spaces}
\label{tab:scalar_elliptic_simple}
\end{table}

In Table~\ref{tab:scalar_elliptic_var_thresh} we show the number of PCG-iterations and condition numbers for two preconditioners based on spectral coarse spaces. In columns $2 - 4$ we report the results for a coarse space of fixed dimension $162$ and the threshold for which this is achieved. 
In columns $5 - 7$ we present the results for a fixed threshold $1/\tau_\lambda=0.5$. We note that the difference in the performance is only for values of the contrast below $1e4$.

\begin{table}
\begin{tabular}
{V{2.5} c V{2.5} c| c|c  V{2.5} c | c | c  V{2.5}} \clineB{2-7}{2.5}
\multicolumn{1}{  c V{2.5}}{}  & \multicolumn{3}{| c V{2.5}}{Spectral coarse space of dim $\qVscrH=162$} & \multicolumn{3}{ c V{2.5}}{Spectral coarse space, $1/\tau_\lambda=0.5$} \\ \hlineB{2.5}
$\qkappamax/\qkappamin$ & \# iter. & cond. & $1/\tau_\lambda$ & \# iter  & cond. & dim $\qVscrH$ \\ \hlineB{2.5}
$1e2$ & $18$ & $7.45$ & $1.39$  & 25 & 15.6 & 76  \\ \hline
$1e3$ & $17$ & $5.99$ & $0.92$  & 21 & 11.5 & 145 \\ \hline
$1e4$ & $18$ & $6.20$ & $0.50$  & 18 & 6.20 & 162 \\ \hline
$1e5$ & $18$ & $6.18$ & $0.50$  & 18 & 6.18 & 162 \\ \hline
$1e6$ & $19$ & $6.19$ & $0.50$  & 19 & 6.19 & 162 \\ 
\hlineB{2.5}
\end{tabular}\\[2ex] 
\caption{Elliptic Problem of Second Order: Numerical results for spectral coarse spaces $\qVscrH$ of fixed dimension and fixed threshold $1/\tau_\lambda=0.5$.}
\label{tab:scalar_elliptic_var_thresh}
\end{table}

\subsubsection{Scalar Elliptic Problem -- Mixed Formulation (see Section~\ref{sec:Darcy})}
\label{sec:Darcy_numeric}
Here, the finite element space is the Raviart-Thomas space of degree $0$ ($RT0$) for the velocity and piecewise constants for the pressure on the same rectangular fine mesh as above. The right hand side $\qbf$ in \eqref{eq:Darcy} is chosen to compensate for the boundary condition of unit flow in $\qx$-direction, i.e., $\qbu\cdot\qbn = \qbe_1 \cdot \qbn$ on $\qdOmega$, where $\qbe_1$ is the first Cartesian unit vector. The (divergence free) coarse velocity space is constructed as outlined in Remark~\ref{rem:stab_decom_mixed}. We first construct a basis of the spectral coarse space corresponding to the stream function space. The corresponding coarse velocity space is then given by the span of the curl of these basis functions. Note, that the stream function space corresponding to $RT0$ is given by the space of Lagrange polynomials of degree $1$ (see \cite[Section~4.4]{GR86}).

As partition of unity $\{\qxi_\qj\}_{\qj=1}^{\qnbx}$ we could simply use the bilinear Lagrange basis functions corresponding to the coarse mesh $\qTcalH$. Nevertheless, for consistency with the Brinkman case (see Section~\ref{sec:Brinkman_numeric}), where we have higher regularity requirements, we choose the $\qxi_\qj$'s to be piecewise polynomials of degree 3, such that all first derivatives and the lowest mixed derivatives are continuous  and $\qxi_\qj(\qbx_\qi)=\delta_{i,j}$ for $\qi,\qj=1,\ldots,\qnbx$.

In Table~\ref{tab:scalar_Darcy} we present the numerical results for this problem and Geometry~1 (see Figure~\ref{fig:geometry}).
The dimension of the fine space is 12416.
In columns $2 - 4$ we report the number of iterations, the size of the standard coarse space, 
and the condition number of the preconditioned system. Here, the standard coarse (velocity) space is given by the span of the curl of the partition of unity functions corresponding to interior coarse mesh nodes. Columns $5 - 7$ contain the number of iterations, the dimension of the coarse space, as well as the condition number of the preconditioned system. It is clear that for the standard coarse space of dimension $49$ 
the condition number grows with increasing the contrast and so does the number of iterations. However, when the coarse space includes all coarse eigenfunctions below the threshold $1/\tau_\lambda=0.5$, the preconditioner shows convergence rates and condition numbers independent of the contrast.


\begin{table}
\begin{tabular}{V{2.5} c V{2.5} c| c| c V{2.5} c |  c|c V{2.5} }\clineB{2-7}{2.5}
\multicolumn{1}{ c V{2.5}}{} &  \multicolumn{3}{ c V{2.5}}{Standard coarse space $\qVscrH^{st}$} &  \multicolumn{3}{c V{2.5}}{Spectral coarse space $\qVscrH$} 
\\ \hlineB{2.5}
$\qkappamax/\qkappamin$ & \# iter. & dim $\qVscrH^{st}$ & cond.\ num.& \# iter. & dim $\qVscrH$ & cond.\ num.\ \\ \hlineB{2.5}
$1e2$ & 32 &  49 &  2.87e1 & $23$  &  $86$ & $13.87$ \\ \hline
$1e3$ & 50 &  49 &  2.26e2 & $24$  & $129$ & $18.38$ \\ \hline
$1e4$ & 63 &  49 &  2.19e3 & $17$  & $162$ & $6.57$ \\ \hline
$1e5$ & 80 &  49 &  2.18e4 & $18$  & $162$ & $6.65$ \\ \hline
$1e6$ & 87 &  49 &  2.13e5 & $19$  & $162$ & $6.68$ \\ \hlineB{2.5}
\end{tabular}\\[2ex] 
\caption{Scalar elliptic equation in mixed formulation: numerical results for Standard coarse space $\qVscrH^{st}$ and spectral coarse spaces $\qVscrH$ obtained for a fixed threshold $1/\tau_\lambda=0.5$.}
\label{tab:scalar_Darcy}
\end{table}

\subsubsection{Brinkman Problem (see Section~\ref{sec:Brinkman})}
\label{sec:Brinkman_numeric}
Next, we present the numerical experiments for the Brinkman problem \eqref{eq:Brinkman}, where the right hand side $\qbf$ is chosen to compensate for the boundary condition of unit flow in $\qx$-direction, i.e., $\qbu = \qbe_1$ on $\qdOmega$. The viscosity $\qmu$ is chosen to be $0.01$ and $\qkappa$ varies depending on the contrast (see Table~\ref{tab:Brinkman}).

We discretize this problem with an $H(div)$-conforming Discontinuous Galerkin discretization (cf.\ \cite{WY07,Wil09}) using Raviart-Thomas finite elements of degree 1 (RT1). We again employ a $64\times64$ fine grid. It is well-known (see \cite[Section~4.4]{GR86}) that in two spatial dimensions the stream function space corresponding to the RT1 space is given by Lagrange biquadratic finite elements. For a generalization to three spatial dimensions one has to utilize N\'ed\'elec elements of appropriate degree. As above, we use an $8\times8$ coarse mesh. We choose $\{\qxi_\qj\}_{\qj=1}^{\qnbx}$ as described in Section~\ref{sec:Darcy_numeric}, which satisfies all regularity constraints.

In Table \ref{tab:Brinkman} we give the number of iterations, the dimension of the coarse space in the additive Schwarz preconditioner, as well as the estimated condition number of the preconditioned system. The dimension of the fine space is 49408. As for the scalar elliptic case in mixed formulation, the coarse (divergence free) velocity space is constructed as outlined in Remark~\ref{rem:stab_decom_mixed}. In columns $2 - 5$ we present the results for the case of the standard coarse space of dimension $49$, which as in Section~\ref{sec:Darcy_numeric} is given by the span of the curl of the partition of unity functions corresponding to interior coarse mesh nodes.

We observe, that the increase in the contrast leads to an increase in the condition number and subsequently to an increase of the number of iterations. Further, in columns $5 - 7$ we report the number of iterations, the dimension of the coarse space in the additive Schwarz preconditioner, as well as the estimated condition number of the preconditioned system for the spectral coarse space obtained by a fixed threshold $1/\tau_\lambda=0.5$. For the Brinkman problem the performance of the preconditioner is  also robust. We should note however, that Brinkman's equation is much more difficult to solve due to the fact that the overall system of linear equations is a saddle point problem.

\begin{table}
\begin{tabular}{V{2.5} c V{2.5} c| c| c V{2.5} c |  c| c V{2.5} }\clineB{2-7}{2.5}
\multicolumn{1}{ c V{2.5}}{} &  \multicolumn{3}{ c V{2.5}}{Standard coarse space \mbox{$\qVscrH^{st}$}} &  \multicolumn{3}{c V{2.5}}{spectral coarse space \mbox{$\qVscrH$}}
\\[1ex] \hlineB{2.5}
$\qkappamax/\qkappamin$ & \# iter. & dim. \mbox{$\qVscrH^{st}$} & cond.\ num.& \# iter. & dim. \mbox{$\qVscrH$} & cond.\ num.\ \\ \hlineB{2.5}
$1e2$ & 27 &  49 &  2.13e1 & $25$  &  $60$ & $14.69$ \\ \hline
$1e3$ & 39 &  49 &  4.22e2 & $28$  & $ 75$ & $21.73$ \\ \hline
$1e4$ & 70 &  49 &  2.25e3 & $29$  & $106$ & $21.83$ \\ \hline
$1e5$ & 91 &  49 &  1.51e4 & $24$  & $153$ & $13.08$ \\ \hline
$1e6$ & 113 & 49 &  1.24e5 & $22$  & $164$ & $13.82$ \\ \hlineB{2.5}
\end{tabular}\\[2ex] 
\caption{Numerical results for Brinkman's equation using standard coarse space $\qVscrH^{st}$ and spectral coarse spaces $\qVscrH$.}
\label{tab:Brinkman}
\end{table}


\subsection{Numerical experiments for Geometry 2 in Figure \ref{fig:geometry}}

These numerical experiments are aimed to compare the performance of the iterative method applied to the {\it second order elliptic problem} in Galerkin formulation (see Section~\ref{sec:ScalarElliptic}) for a permeability given in Geometry 2 (see Figure~\ref{fig:geometry}). 
The goal here is to demonstrate the  coarse space  dimension reduction when 
using multiscale partition of unity functions instead of 
standard ones.  The dimension of the fine-grid space is 4225. 

In Table~\ref{tab:scalar_elliptic} we present the results when the preconditioner is based on the spectral coarse space $\qVscrH$ (columns $ 2 - 4$), the multiscale spectral coarse space $ \qVscrtH$ (columns $5 - 7$), and the multiscale coarse space $ \qVscrH^{ms}$ (columns $8 - 10$).
Comparing the data for the spaces $\qVscrH$ and $\widetilde \qVscrH$ shows that  the number of PCG-iterations and the estimated condition number of the preconditioned system are robust with respect to the contrast $\qkappamax/\qkappamin$. 
We can also observe that when using the spectral coarse space $ \qVscrH$ the dimension of the coarse space increases as the contrast increases $\qkappamax/\qkappamin$, which is in agreement with the analysis of Section \ref{sec:ScalarElliptic}. The decrease in the estimated condition number when going from $\qkappamax/\qkappamin=1e2$ to $\qkappamax/\qkappamin=1e3$ and further to $\qkappamax/\qkappamin=1e4$ can be explained by the fact that for higher contrasts more eigenvalues are below the prescribed threshold, yielding a higher dimensional coarse space and a lower condition number. However, it is important to note that the dimension of the coarse space reaches a maximum for $\qkappamax/\qkappamin$ above a certain threshold. As we can see 
for $\qkappamax/\qkappamin$ in the range $1e4,\,\ldots,\, 1e6$, the dimension of the coarse space stays the same. By the analysis in Section \ref{sec:ScalarElliptic} we know that there is only a finite number of asymptotically small (with the contrast tending to infinity) generalized eigenvalues. The reported data can be seen as evidence that for this specific configuration we have reached this asymptotic regime for $\qkappamax/\qkappamin=1e4$.

\begin{table}
\begin{tabular}{V{2.5} c V{2.5} c| c| c V{2.5} c |  c |c V{2.5} c |  c |c V{2.5} }\clineB{2-10}{2.5}
\multicolumn{1}{ c V{2.5}}{} &  \multicolumn{3}{ c V{2.5}}{ \mbox{$\qVscrH$}} &  \multicolumn{3}{ c V{2.5}}{\ $\qVscrtH$} &  \multicolumn{3}{ c V{2.5}}{ \mbox{$\qVscrH^{ms}$}}  \\ \hlineB{2.5}
$\frac{\qkappamax}{\qkappamin}$ & \# iter. & dim $\qVscrH$ & cond.\ \#& \# iter. & dim \mbox{$\qVscrtH$} & cond.\ \#\ &  \# iter& dim $\qVscrH^{ms}$ & cond \#  \\ \hlineB{2.5}
$1e2$ & 22 &  163 & 12.15 & $21$  &  $44$ & $10.81$ & 19 &  49 & 8.70\\ \hline
$1e3$ & 18 &  612 &  8.42 & $20$  & $ 60$ & $9.86$  & 35 &  49 & 5.97e1 \\ \hline
$1e4$ & 15 &  838 &  4.92 & $22$  &  $60$ & $10.90$ & 44 &  49 & 5.63e2 \\ \hline
$1e5$ & 16 &  838 &  4.92 & $22$  & $60$ & $11.01$  & 53 &  49 & 5.59e3 \\ \hline
$1e6$ & 17 &  838 &  4.92 & $22$  & $60$ & $11.01$  & 66 &  49 & 5.59e4 \\ \hlineB{2.5}
\end{tabular}\\[2ex] 
\caption{Scalar elliptic -- Galerkin formulation: results for spectral coarse space $\qVscrH$, multiscale spectral coarse space $\widetilde \qVscrH$, and multiscale coarse space $ \qVscrH^{ms}$.}
\label{tab:scalar_elliptic}
\end{table}

In columns $8 - 10$ we present the numerical results of the algorithm when the preconditioner is based on the
 multiscale coarse space $\qVscrH^{ms}$, 
which  consist of one 
basis function per interior coarse node.
As we can see from the data, 
the number of PCG-iterations as well as the condition number of the preconditioned system grow steadily with the growth of the contrast. 

The important point to observe when using the multiscale spectral coarse space $\qVscrtH$ (see columns $5-7$ of Table~\ref{tab:scalar_elliptic}) is that its dimension is drastically reduced compared to the spectral coarse space $\qVscrH$. In our specific example, the largest dimension of the multiscale spectral coarse space $\widetilde \qVscrH$, which is constructed using the multiscale partition of unity $\{\qxit_\qj\}_{\qj=1}^{\qnbx}$, is $60$, compared to the dimension $838$ of the spectral coarse space $\qVscrH$, which is based on the standard partition of 
unity $\{\qxi_\qj\}_{\qj=1}^{\qnbx}$. 
0ne is generally interested in keeping 
the dimension of the coarse space 
as small as possible, especially when 
the problem is solved multiple times.
The data is a confirmation of our reasoning in Section~\ref{sec:dim_coarse}.

\subsection{Numerical experiments for Geometries 3 and 4} 
In Table~\ref{tab:scalar_elliptic_mult_256} we present the numerical results for the scalar elliptic equation of 
second order in Galerkin formulation 
from Section \ref{sec:ScalarElliptic} for 
highly heterogeneous permeability 
distributions shown in Figure~\ref{fig:geom_256}. Geometry~3 represents a rather challenging example: the permeability field is highly heterogeneous with more than 4000 small and about 100 large randomly distributed inclusions. We consider this a challenging  test for the robustness of the iterative method by performing a relatively small number of iterations using a coarse space of low dimension.
Here, we have used  a  $16 \times 16$ coarse mesh  and subdivided each coarse cell into  $16 \times 16$ subcells to obtain a $256 \times 256$ fine mesh. The preconditioner is based on the multi-scale spectral coarse space $\widetilde \qVscrH$.  The dimension of the fine space is 66049 while dimension of the coarse space is at most 293.  As we can see, the condition number of the preconditioned system is robust with respect to the contrast and the dimension of the coarse space is quite small. 

Geometry~4 (see Figure~\ref{fig:geom_256}) represents a more challenging problem 
in that it is no longer  a binary medium, i.e., $\qkappa$ assumes many and not just two extreme values. The geometry is generated by setting $\qkappa$ in a fine mesh cell to $10^{\qgamma\, \text{rand}}$, where $\text{rand}$ denotes a uniformly distributed random number and $\qgamma=2,\ldots, 6$. 
This  produces a random field $10^{\qgamma\, \eta(\qbx)}$ where
$\eta(\qbx)$ is a realization of a  
spatially uncorrelated random field. 
This yields a ``background'' on top of which we put randomly generated inclusions similar to Geometry~3. In Table~\ref{tab:scalar_elliptic_mult_256} (columns $5-7$) we report the numerical results using the preconditioner based on the multiscale spectral space $\qVscrtH$. As we can see, the number of PCG iterations and the condition number of the preconditioned system are robust with respect to increases in the contrast. It is furthermore important to note that, even for this random case, the dimension of the coarse space stays reasonably small (at most 397) compared to the dimension of the fine space, i.e., 66049. This exemplifies the robustness and applicability of the numerical method developed above.

\begin{figure}
\subfigure[Geometry 3: periodic background and randomly distributed inclusions.]{
\includegraphics[width=.34\textwidth]{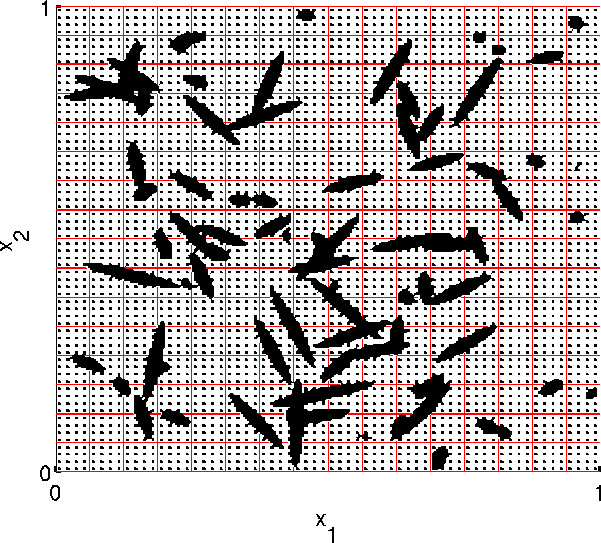}
}
\hspace{1cm}
\subfigure[Geometry 4: random background and randomly distributed inclusions -- logarithmic plot of $\qkappa$.]{
\includegraphics[width=.4\textwidth]{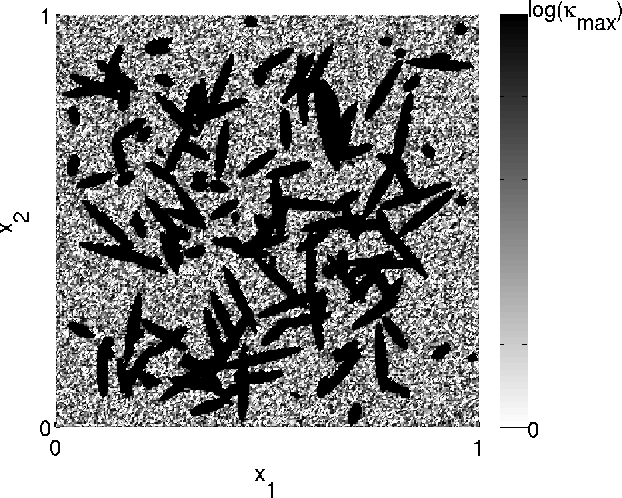}
}
\caption{Two geometries with a $256\times 256$ fine mesh and a $16\times 16$ coarse mesh.
}
\label{fig:geom_256}
\end{figure}

\vspace{1cm}
\begin{table}
\begin{tabular}{V{2.5} c V{2.5} c| c|c  V{2.5} c | c | c  V{2.5}} \clineB{2-7}{2.5}
\multicolumn{1}{ c V{2.5}}{}  & \multicolumn{3}{| c V{2.5}}{Geometry 3} & \multicolumn{3}{ c V{2.5}}{Geometry 4} \\ \hlineB{2.5} 
$\qkappamax/ \qkappamin$ & \# iter.\ & dim $\qVscrtH$ & cond.\ \# & \# iter.\ & dim $\qVscrtH$ & cond.\ \#\\ \hlineB{2.5}
$1e2$ & $22$ & $209$ & $11.2$ & $19$ & $217$ & $8.47$\\ \hline
$1e3$ & $24$ & $259$ & $15.2$ & $20$ & $221$ & $9.48$\\ \hline
$1e4$ & $23$ & $275$ & $11.2$ & $22$ & $244$ & $11.5$\\ \hline
$1e5$ & $24$ & $277$ & $11.2$ & $25$ & $317$ & $16.3$\\ \hline
$1e6$ & $27$ & $293$ & $11.7$ & $23$ & $397$ & $11.7$\\ \hlineB{2.5}
\end{tabular}\\[2ex] 
\caption{Scalar elliptic problem -- Galerkin formulation: Numerical results for permeability fields shown in Figure~\ref{fig:geom_256} using multiscale spectral coarse spaces $\qVscrtH$.}
\label{tab:scalar_elliptic_mult_256}
\end{table}

\section{Conclusions}
The theory developed above introduces a method for constructing stable decompositions with respect to symmetric positive definite operators. The robustness with respect to problem and mesh parameters is proved under rather general assumptions. We have furthermore applied this abstract framework to several important cases, i.e., the scalar elliptic equation in Galerkin and mixed formulation, Stokes' equations, and Brinkman's equations. For the scalar elliptic equation in Galerkin formulation, we have additionally presented a strategy of reducing the dimension of the coarse space in the stable decomposition. To verify our analytical results, we have performed several numerical experiments, which are in coherence with our theory and show the usefulness of the method.

\section*{Acknowledgments}
The research of Y. Efendiev, J. Galvis, and R.~Lazarov was supported in parts by award KUS-C1-016-04, made by King Abdullah University of Science and Technology (KAUST). The research of R. Lazarov and J. Willems was supported in parts by NSF Grant DMS-1016525.

\bibliographystyle{plain}
\bibliography{references,referencesJuan,bib_raytcho}

\end{document}